\documentclass[11pt,a4paper]{amsart}
\usepackage[foot]{amsaddr}
\usepackage{amsmath}
\usepackage{amsthm}
\usepackage{amssymb}
\usepackage{mathrsfs}
\usepackage{verbatim}
\usepackage{xcolor}
\usepackage[colorlinks=true,linkcolor=blue,citecolor=red]{hyperref}
\usepackage{enumitem}
\usepackage{multicol}
\usepackage{graphicx}
\usepackage{float}
\usepackage[nameinlink]{cleveref}

\newcommand{\etalchar}[1]{$^{#1}$}

\newtheorem{theorem}{Theorem}
\newtheorem*{theorem*}{Theorem}
\newtheorem{lemma}{Lemma}
\newtheorem{proposition}[theorem]{Proposition}
\newtheorem{definition}{Definition}

\newtheorem{example}{Example}

\newcommand{\RR}{\mathbb{R}}

\newcommand{\NN}{\mathbb{N}}
\newcommand{\ZZ}{\mathbb{Z}}

\newcommand{\D}{\Delta}
\newcommand{\test}{C_c^{\infty}}

\newcommand{\Om}{\Omega}
\newcommand{\II}{\iint\limits_}
\newcommand{\I}{\int\limits_}

\numberwithin{equation}{section}
\numberwithin{theorem}{section}
\numberwithin{lemma}{section}
\numberwithin{example}{section}
\allowdisplaybreaks

\makeatletter
\@namedef{subjclassname@2020}{\textup{2020} Mathematics Subject Classification}
\makeatother

\begin{document}
    
    \title[Hardy and Poincar\'e inequalities]{Hardy and Poincar\'e inequalities in fractional Orlicz-Sobolev spaces}
    
    \author{Kaushik Bal}
    
    \email{kaushik@iitk.ac.in}
    
    \author{Kaushik Mohanta}
    \email{kmohanta@iitk.ac.in}
    \author{Prosenjit Roy}
    
    \email{prosenjit@iitk.ac.in}
    
    \author{Firoj Sk}
    
    \email{firoj@iitk.ac.in}
    
    \address{Indian Institute of Technology Kanpur, India}

    .
    \date{}

    \smallskip
    \begin{abstract}
     We provide sufficient conditions for boundary Hardy inequality to hold in bounded Lipschitz domains, complement of a point (the so called point Hardy inequality), domain above the graph of a Lipschitz function, complement of a bounded Lipschitz domain in fractional Orlicz-Sobolev setting. As a consequence we get sufficient conditions for regional fractional Orlicz Poincar\'e inequality in bounded Lipschitz domains. Necessary conditions for fractional Orlicz Hardy and regional fractional Orlicz Poincar\'e inequalities are also given for bounded Lipschitz domains. Various sufficient conditions on open sets are provided for fractional Orlicz Poincar\'e inequality and regional fractional Orlicz Poincar\'e inequality  to hold.
    \end{abstract}
    
    \maketitle
    

\section{Introduction}
The aim of this article is to study two very well known inequalities, the Poincar\'e inequality and the Hardy inequality on fractional Orlicz-Sobolev setting. The classical Poincar\'e inequality  \cite[Chapter~5.8.1]{evans} states that for any bounded domain $D \subseteq\RR^N$, $q\geq1$ there exists $c=c(q,N,D)>0$ such that for any $f\in\test(D)$,
\begin{equation}\label{poin}
\|f\|_{L^q(D)} \leq c\|\nabla f\|_{L^q(D)}.
\end{equation}
        For bounded domains it is a standard fact that the best constant is attained.
        Although generalized in many different direction as can be seen in \cite{beb,Jeri,KeZh,lieb}, for the purpose of this paper we would like to refer Gossez \cite[Lemma 5.7]{Gos}, who generalized \cref{poin} in Orlicz-Sobolev setting (defined below).	The standard boundary Hardy inequality states that if $D$ is a
        bounded Lipschitz domain in $\RR^N$, $q\geq1$, then there exists a constant
        $c=c(q)>0$ such that for $f\in\test( D )$,
        \begin{equation}\label{har}
        \|f/\delta_D\|_{L^q(D)} \leq c\|\nabla f\|_{L^q(D)},
        \end{equation}
        where $ \delta_D(x) := \inf\lbrace |x-y|:y\in
        D^c\rbrace $. We shall frequently denote $\delta_D(x)$ by $\delta_x$.  This has been further developed in \cite{BrMa, BrMaSh, kufner,MoMiPi} and several other works. For the generalization to Orlicz version  of local Hardy inequality we refer to \cite{Heinig,KaPi,MaMiOhShi,MiNaOhSh}.\smallskip
        
        We shall primarily be concerned with the ``Orlicz space" $L^A(D)$ and the ``fractional Orlicz-Sobolev space" $W^{s,A}(D)$. We start by defining these spaces. A continuous, convex function $A:[0,\infty)\to[0,\infty) $ such that
        $\lim\limits_{t\to0}\frac{A(t)}{t}=0$ and
        $\lim\limits_{t\to\infty}\frac{A(t)}{t}=\infty$ is called an
        \textbf{$N$-function} or an \textbf{Young function} \protect{\cite[Chapter 8.2]{ada}}.  
        
        \begin{definition}[ \protect{\cite[Chapter 8.2]{ada}}]\label{delta2 cond}
            We say that $A$ satisfies the $\mathbf{\D_2}$ \textbf{condition globally} or
            simply the $\mathbf{\Delta_2}$ \textbf{condition} ($A\in\D_2$) if there exists constant $p>0$ such
            that
            $$
            A(2t)\leq pA(t), \ \forall \ t> 0.
            $$
        \end{definition}
        \noindent	The set $$L^A( D ):=\left\{ f: D \to \RR \  \textrm{measurable} \ \Big| \ \exists \ \lambda>0 \mbox{ such that } M_{L^A( D )}\left(\frac{f}{\lambda}\right)<\infty\right\} $$ is called the \textbf{Orlicz space} and the set $$W^{s,A}( D ):=\left\{ f\in L^A(D) \ \Big| \ \exists \ \lambda>0 \mbox{ such that } M_{W^{s,A}( D )}\left(\frac{f}{\lambda}\right)<\infty\right\}$$ is called the \textbf{fractional Orlicz-Sobolev space}, where $$ M_{L^A( D )}(f):= \I{ D } A(|f(x)|) dx \  \textrm{ and } M_{W^{s,A}( D )}(f):= \II{ D \times D }A\left( \frac{| f(x)-f(y) |}{| x-y |^s}\right) \frac{dxdy}{|x-y|^N}. $$
        \noindent In the case $A(t)=t^q$ for some $q>1$, $L^A(D)$ and $W^{s,A}(D)$ are well known Lebesgue space $L^q(D)$ and the fractional Sobolev space $W^{s,q}(D)$ respectively \cite[p.~524]{hhg}.
        \smallskip
        
        For $s \in (0,1)$, we start by defining the following important quotients that will be used frequently throughout this article:
        $$
        H_{N,s,A}( D ):=\inf_{\substack{f\in\test( D )\\f\neq 0
        }}\frac{M_{W^{s,A}( D )}(f)}{\I D A\left(
            \frac{|f(x)|}{\delta_x^s}\right)dx}, \ P^1_{N,s,A}( D ):=\inf_{\substack{f\in\test( D )\\f\neq 0
        }}\frac{M_{W^{s,A}( D )}(f)}{\I D A(|f(x)|)dx}
        $$
        and
        $$
        P^2_{N,s,A}( D ):=\inf_{\substack{f\in\test( D )\\f\neq 0
        }}\frac{M_{W^{s,A}( \RR^N )}(f)}{\I D A(|f(x)|)dx}.
        $$
        \noindent  We shall say that 
        \begin{itemize}
            \item 
            fractional Orlicz Hardy inequality (denoted as $FOHI(s,A)$)
            holds if $H_{N,s,A}( D )>0$;
            \item
            regional fractional Orlicz Poincar\'e inequality ($RFOPI(s,A)$) holds if $P^1_{N,s,A}( D )>0$;
            \item
            fractional Orlicz Poincar\'e inequality ($FOPI(s,A)$) holds if $P^2_{N,s,A}( D )>0$.
            
        \end{itemize}
        \smallskip
        
        We start by recalling some literature on  fractional Orlicz Hardy inequality for the case when $A(t) = t^q, \  q>1$. Kufner \cite{kufner3} proved $FOHI(s,t^q)$ in one
        dimension when $ D =(0,\infty)$.  For bounded Lipschitz domain $D$, Dyda \cite[Theorem~1.1 and Section~2]{dyd}
        proved the following results:
        \begin{theorem}[Dyda]\label{dyda's theorem}
            Let $\beta>0$ and $q>1$.  The Hardy inequality, 
            \begin{equation}\label{dyda's inequality}
            \I{D}\frac{|u(x)|^q}{\delta_D(x)^\beta}dx\leq c\ \II{D\times D}\frac{|u(x)-u(y)|^q}{|x-y|^{N+\beta}}dxdy,\quad \mbox{for all } u\in C_c(D),
            \end{equation}
            where $c=c(D.\beta,N,q)<\infty$ is a constant, holds true in each of the following cases:
            \begin{enumerate}
             \item[(T1)] $D$ is a bounded Lipschitz domain and $\beta>1$;
             \item[(T2)] $D$ is a complement of a bounded Lipschitz domain, $\beta\neq 1$ and $\beta\neq N$;
             \item[(T3)] $D$ is a domain above the graph of a Lipschitz function $\RR^{N-1}\to\RR$ and $\beta\neq1$;
             \item[(T4)] $D$ is a complement of a point and $\beta\neq N$.
            \end{enumerate}
        \end{theorem}
    \noindent The question of best constant in fractional Hardy inequality, that is  the exact value of $FOHI(s,t^q)$, was first addressed in \cite[Theorem~1]{bogdan} for upper half space. Heinig \textit{et al.} \cite[Theorem~3.1]{kufner4} and Kufner \textit{et al.} \cite[Theorem~5.23]{kufner2} studied one dimensional $FOHI(s,t^q)$ between two weighted $L^q$ spaces. Reader may refer to \cite{brasco, chen, dyda4, dyda3, dyda2, edmunds, frank1, frank, kufner4, hurri, ihnatsyeva, loss} and references therein for more information related to fractional Hardy inequality. \smallskip
        
        Concerning  $RFOPI(s,A)$ and $FOPI(s, A)$, we start with the trivial observation that\linebreak
        $P_{N,s,A }^2(D) \geq P_{N,s,A }^1(D),$
        that  is $FOPI(s,A)$ holds whenever $RFOPI(s,A)$ is true. For a bounded domain $D$,  $RFOPI(s, t^2)$ is true  if and only if $2s>1$ \cite[Proposition~3.2]{chen2}  whereas $FOPI(s,t^2)$ holds for all $s\in (0,1)$ \cite[Theorem~1.2]{Ind-Roy}. It was first established in \cite[Lemma~1]{yeressian} that $FOPI(s,t^2)$ is true for all values of $s$, if the domain is an infinite strip i.e. $D = (0,1) \times \mathbb{R}^{N-1}$, though the best constant was not established. The best constant for the above case is obtained in \cite[Theorem~1.3]{Ind-Roy}. $RFOPI(s, t^q)$ for strip like domain is  further studied in \cite[Theorem~1.1]{Ind}, where it is established that $RFOPI(s,t^2)$ is true if and only if $2s >1$. Sufficient criteria on domains for $RFOPI(s, t^2)$ and $FOPI(s, t^2)$ to hold, is provided in \cite[Theorem~1.1-1.4]{Ind} and \cite[Theorem~1.2]{Ind-Roy} respectively. For other related works on fractional Poincar\'e inequality and its applications, we refer the reader to \cite{chen,Ind-Roy,chowdhury}, and the references therein.\smallskip
        
        As of now the theory of fractional Orlicz-Sobolev spaces is still at an early stage of development. Bonder-Salort \cite[Section~2.2]{Bon} defined the space $W^{s,A}(\RR^N)$ and established basic properties of the space. For an open bounded set $D\subseteq\RR^N$, the space $\lbrace f\in L^{A}(D) \ | \ M_{W^{s,A}(\RR^N)}(f)<\infty \rbrace$ and corresponding fractional Laplacian have been considered in \cite{salort2,salort}. The fractional Orlicz-Sobolev space $W^{s,A}(D)$ has appeared in Bahrouni \textit{et al.} \cite[Section~2]{Bahrouni} for any open bounded set $D\subseteq\RR^N$. For other works on fractional Orlicz-Sobolev spaces one may refer to \cite{AlCiPiSl1,AlCiPiSl2,AzBeSr,BaOu,BaOuTa,BaSa,CoPa,FeHaRi,MaSaVe}. Salort \cite[Theorem~1.1]{sal} proved $FOHI$ in one dimension. In \cite[Theorem~7.4]{AlCiPiSl}, the authors proved $FOHI$ in the special case where the domain is $\RR^N$. Fractional Orlicz Hardy inequality and regional fractional Orlicz Poincar\'e inequality for any domain or fractional Orlicz Poincar\'e inequality for unbounded domains have not yet appeared in the literature (apart from \cite{sal,AlCiPiSl} as mentioned above). For $A\in\D_2$, an open bounded set $D$, $s\in(0,1)$, the following version of $FOPI(s,A)$ is proved in \cite[Theorem~6.1 and Corollary~6.2]{Bon}:
        $$
        M_{L^{A}(D)}(f)\leq c M_{W^{s,A}(\RR^N)}(f), \hspace{3mm} f \in L^A(D)
        $$
        where $c=c(A,s,N)>0$. Similar results are also proved in \cite{salort2,salort} for general $N$-function (without $\D_2$ condition).
        \smallskip
        
        Primary  aim of this article is to  generalize \cref{dyda's theorem} in the fractional Orlicz setting. This is achieved via \cref{Main theorem,Main Theorem Point Hardy,,Main Theorem Other Hardy}. \Cref{Main theorem} also provides sufficient condition on $A$ and $s$ for $RFOPI(s,A)$ to hold on any bounded Lipschitz domain, whereas \cref{Main theorem not true} gives sufficient condition, on bounded Lipschitz domains, for $FOHI(s,A)$ and $RFOPI(s,A)$ not to hold. As an application of \cref{Main theorem,Main theorem not true}, at the end of \cref{section hardy} (see \cref{table}), we give several examples of $A\in\D_2$ and respective ranges of $s$ for which Hardy and regional Poincar\'e inequality is true or false. In particular for $N$-functions $t^q,\ q>1$  and  $(1+t)\log(1+t) -t$ a complete answer, for all values of $s \in (0,1)$ can be provided for $FOHI(s,A)$. At this point, we would like to mention that the line of argument in the proofs of \cref{Main theorem,Main Theorem Point Hardy,,Main Theorem Other Hardy} are adapted from \cite{dyd} where Dyda, in fact, predicted the possibility of his methods being generalized. Secondary objective of this article is to study $RFOPI(s,A)$ and $FOPI(s,A)$ for unbounded domains in $\RR^N$. In this direction our results are \cref{1-dim poincare,Sufficient condition}. In \cref{1-dim poincare} we give complete characterization of domains in 1-dimension for which $RFOPI(s,A)$ holds, provided $\lim\limits_{\lambda\to0+}\alpha_{s,A}(\lambda)=0$ (see \cref{alpha fnc}). \Cref{Sufficient condition} provides different sufficient criteria on domains for $FOPI(s,A)$ and $RFOPI(s,A)$ to hold. As an application of \cref{Sufficient condition}, at the end of \cref{section poincare} we provide several non-trivial examples of domains for which $RFOPI(s,A)$ and $FOPI(s,A)$ holds. An adapted version of a change of variable formula introduced in \cite[Lemma~2.4]{loss} is the key ingredient in the proof of \cref{Sufficient condition} which can be regarded as the fractional Orlicz analogue of the results obtained in \cite[Theorem~1.2]{Ind-Roy} and \cite[Theorem~1.3]{Ind}.\smallskip
        
         Our main results are stated below. \smallskip
        
        \begin{theorem}\label{Main theorem}
            Let $D\subseteq\RR^N$ be a bounded Lipschitz domain, $A\in\D_2$, $s\in(0,1)$ and\\
             $\liminf\limits_{\lambda\to
                0+}\alpha_{s,A}(\lambda)=0 $, where $\alpha_{s,A}:[0,\infty)\to\RR$ is defined by
            \begin{equation}\label{alpha fnc}
            \alpha_{s,A}(\lambda):=\sup\limits_{t\in[0,\infty)}\frac{A(\lambda
                t)}{\lambda^{\frac{1}{s}}A(t)}.
            \end{equation}
            Then $FOHI(s,A)$ and $RFOPI(s,A)$ holds, that is $H_{N,s,A}(D)>0$ and
            $P^1_{N,s,A}(D)>0$.
        \end{theorem}
        
        \begin{theorem}\label{Main theorem not true}
            Let $D\subseteq\RR^N$ be a bounded Lipschitz domain, $A$ be an $N$-function and
            \begin{equation}\label{integralcond} 
            \lim\limits_{\varepsilon\to0+}\varepsilon\I{0}^{\varepsilon^{-s}}\frac{A(z)}{z}dz=\beta\in\RR.
            \end{equation}
            Then
            \begin{enumerate}
                \item if $\beta=0$, then both $FOHI(s,A)$ and $RFOPI(s,A)$ do not hold,
                \item if $\beta\in(0,\infty)$, then $FOHI(s,A)$ does not hold.
            \end{enumerate}
        \end{theorem}
        
        \begin{theorem}\label{Main Theorem Point Hardy}
            Let $s\in(0,1)$, $A\in\D_2$ and $\alpha_{s,A}$ be as in \cref{alpha fnc}. Assume $\liminf\limits_{\lambda\to0+}\lambda^{\frac{1-N}{s}}\alpha_{s,A}(\lambda)=0$ or $\liminf\limits_{\lambda\to\infty}\lambda^{\frac{1-N}{s}}\alpha_{s,A}(\lambda)=0$.
            Then $FOHI(s,A)$ holds in $\RR^N\setminus \lbrace 0 \rbrace$, that is $H_{N,s,A}(\RR^N\setminus \lbrace 0 \rbrace)>0$.
        \end{theorem}  
    
    \begin{theorem}\label{Main Theorem Other Hardy}
        Suppose $D\subseteq\RR^N$ be an open set, $A\in\D_2$, $s\in(0,1)$ and $\alpha_{s,A}$ be as in \cref{alpha fnc}. Then $FOHI(s,A)$ holds true, that is $H_{N,s,A}(D)>0$, in each of the following cases:
        \begin{enumerate}
        \item $D=\lbrace (x',x_N)\in\RR^N\ | \ x'\in\RR^{N-1},\ x_N>\Phi(x')\rbrace$, where $\Phi:\RR^{N-1}\to\RR$ is a Lipschitz map and $s,A$ are such that $\liminf\limits_{\lambda\to 0}\alpha_{s,A}(\lambda)=0$ or $\liminf\limits_{\lambda\to \infty}\alpha_{s,A}(\lambda)=0$;
        \item    
        $D^c$ is closure of some bounded Lipschitz domain and $s,A$ are such that $\liminf\limits_{\lambda\to\infty}\lambda^{\frac{1-N}{s}}\alpha_{s,A}(\lambda)=0=\liminf\limits_{\lambda\to0+}\alpha_{s,A}(\lambda)$
        \quad or \quad
        $\liminf\limits_{\lambda\to0+}\lambda^{\frac{1-N}{s}}\alpha_{s,A}(\lambda)=0$
        \quad or \quad
        $\liminf\limits_{\lambda\to\infty}\alpha_{s,A}(\lambda)=0$.
        \end{enumerate}
\end{theorem}  
\smallskip

Before giving the statement of our last result, that is \cref{Sufficient condition}, we define some terminologies that will be required to formulate it in a precise manner.\smallskip

\begin{definition}
    \label{def:finite ball cond}
    We say that a set $D\subseteq\RR^N$ satisfies \textbf{the finite ball
        condition} if $D$ does not contain arbitrarily large balls, that is if
    $$
    BC( D ):=\sup\lbrace r:\, B(x,r)\subseteq D,\; x\in D\rbrace <\infty.
    $$
\end{definition}

\begin{definition}
    Let $\lbrace D_\beta\rbrace _{\beta}$ be a family of sets in $\mathbb{R}^N$,
    where $\beta\in \Lambda$ (some indexing set). We say that the $\mathbf{FOPI(s,A)}$
    \textbf{holds uniformly for}
    $\mathbf{\lbrace D_\beta\rbrace _{\beta}}$ if
    $\displaystyle\inf_\beta P^2_{N,s,A}(D_\beta) > 0$.
\end{definition}
\noindent Let $\omega\in\mathbb{ S}^{N-1}$ and $x\in \omega^\perp$, define $L_{D}(x,
\omega) := \lbrace t
\ | \ x+t\omega \in D \rbrace\subseteq \RR$.

\begin{definition}
    \label{defn:LS} We say an open
    set $D\subseteq\RR^N$ is of \textbf{type} $\mathbf{LS(s,A)}$ if there exists
    $\Sigma\subseteq\mathbb{S}^{N-1}$ with 
    positive $(N-1)$-dimensional Hausdorff measure, such
    that uniform $FOPI(s,A)$ holds for the family $\lbrace L_D(x,\omega) \rbrace_{\omega
        \in\Sigma,\;x\in
        \omega^\perp}$.
\end{definition} 
        
        \begin{theorem}\label{Sufficient condition}
            Let $D\subseteq\RR^N$ be an open set, $s\in(0,1)$, $A$ be an $N$-function and $\alpha_{s,A}$ be as in \cref{alpha fnc}.
            \begin{enumerate}
                \item Assume that $A\in\D_2$ and  $\lim\limits_{\lambda\to 0+}\alpha_{s,A}(\lambda)=0$. Let there exist $\Sigma\subseteq\mathbb{S}^{N-1}$ with positive $(N-1)$-dimensional Hausdorff measure such that $\sup\limits_{\omega
                    \in\Sigma, x\in
                    \omega^\perp}BC(L_D(x,\omega)) <\infty.$ Then the $RFOPI(s,A)$ holds true in $ D $.
                \item Assume that there exist $R,c_1>0$ such that $\mathcal{L}^{N}(B(x,R)\cap D ^{c})>c_1$
                for
                any $x\in D $. Then $FOPI(s,A)$ holds in $ D $ $\forall s\in(0,1)$.
                \item $FOPI(s,A)$ holds if $D$ is an LS(s,A) domain.
            \end{enumerate}
        \end{theorem}
        
        This paper is arranged in the following way: In \cref{section preli}, some necessary preliminaries are discussed. Proof of \cref{Main theorem,Main theorem not true} followed by some applications are given in \cref{section hardy}. \Cref{Main Theorem Point Hardy,Main Theorem Other Hardy} are proved in \cref{section other hardy}. In \cref{section poincare} we prove \cref{1-dim poincare,Sufficient condition}.

       \bigskip

        \section{Notations and Preliminaries}\label{section preli}
        Throughout the paper the following conventions and notations will be followed, unless mentioned otherwise explicitly:\\
            $D$ will denote an open set in $\mathbb{R}^N$, $s\in(0,1)$, $A$ will denote an $N$-function, $\mathcal{L}^k$ will denote the Lebesgue measure on $\RR^k$, $\mathbb{S}^{k-1}$ will denote the unit sphere in $\RR^k$, $c$ will denote a generic constant which may change from line to line, $X^c$ will stand for the complement of the set $X$ in appropriate universal set (to be understood from the context), for any real valued function $f$ or for any Lipschitz domain $D$, $\mbox{Lip}(f)$ or $\mbox{Lip}(D)$ will denote the Lipschitz constant, $\alpha_{s,A}$ will be as in \cref{alpha fnc}, $p$ will be as in \cref{delta2 cond}.\smallskip
        
        We start with some basic facts about $N$-functions. One may refer to \cite{ada,kra} for detailed discussion on the topic.
        \begin{lemma}[\protect{\cite[Chapter 8.2]{ada}}]\label{N-function}
            $A$ is an $N$-function if and only if there exists a non-decreasing, right continuous function $a:[0,\infty)\to[0,\infty)$ satisfying $a(0)=0, \ a(t)>0 \ \mbox{for} \ t>0, \ \lim\limits_{t\to0+}a(t)=\infty$ such that
            $$
            A(x)=\I0^x a(t)dt.
            $$
        \end{lemma}
        The following two lemmas will be used frequently in the rest of the article.
        \begin{lemma}[\protect{\cite[Chapter~8.2]{ada}}]\label{A-property1}
            Let A be an $N$-function, then $A$  and $t\mapsto\frac{A(t)}{t}$ both are strictly increasing function on $(0,\infty)$.
        \end{lemma}
        
        \begin{lemma}\label{lemma delta2}
            Let $A\in\D_2$. Then
            $$
            A(\lambda t)\leq \lambda^pA(t) \mbox{ for } t\in[0,\infty),\quad \forall\ \lambda\geq1,
            $$
            where $p$ is the constant in \cref{delta2 cond}.
            The above inequality is equivalent to
            $$
            A(\lambda t)\geq \lambda^pA(t) \mbox{ for } t\in[0,\infty),\quad \forall \ 0\leq\lambda\leq1.
            $$
        \end{lemma}
        \begin{proof}
            Since $A\in\D_2$, there exists $p>0$ such that $A(2t)\leq pA(t),\quad \forall \ t>0$. Now for $a$ as in \cref{N-function}, using the non decreasing property of $a$,
            $$
            pA(t)\geq A(2t)=\I0^{2t}a(\tau)d\tau>\I{t}^{2t}a(\tau)d\tau>ta(t).
            $$
            This implies for any $\lambda>1$,
            $$
            \log\left(\frac{A(\lambda t)}{A(t)}\right)=\I{t}^{\lambda t}\frac{a(\tau)}{A(\tau)}d\tau<\I{t}^{\lambda t}\frac{p}{\tau}d\tau=p\log\left(\frac{\lambda t}{t}\right)=\log(\lambda^p).
            $$
            The lemma follows.
        \end{proof}
        
        \begin{lemma}\label{hardy to poincare}
            Let $D\subseteq\RR^N$ be a bounded domain, $A\in\D_2$. Then for some constant $c=c(D,A)>0$,
            $$P^1_{N,s,A}(D)\geq c H_{N,s,A}(D).$$ 
        \end{lemma}
        \begin{proof}
            Let $f\in\test(D)$. First assume  $\mbox{diam}(D)\leq1$. By the monotonicity of $A$,
            $$
            \I{D}A\left( \frac{|f(x)|}{\delta_x^s}\right)  dx\geq \I{D}A(|f(x)|)dx.
            $$
            In the case $\mbox{diam}(D)>1$, we exploit the $\D_2$ condition of
            $A$
            and use \cref{lemma delta2} to get
            $$
            \I{D}A\left( \frac{|f(x)|}{\delta_x^s}\right)  dx\geq
            \I{D}A\left(\mbox{diam}(D)^{-s}|f(x)|\right)dx\geq
            \mbox{diam}(D)^{-sp}\I{D}A(|f(x)|)dx.
            $$
            Hence the lemma follows.
        \end{proof}

        \begin{proposition}\label{monotone-dialtion-fbc}
            Let $s\in(0,1)$ and $p>1$ be as in \cref{lemma delta2}.
            \begin{enumerate}
                \item If $D_1\subseteq D_2\subseteq\RR^N$, then $P^2_{N,s,A}(D_2)\leq P^2_{N,s,A}(D_1).$
                \item Let $ D \subseteq\RR^N$ be an open set and $u\in W^{s,A}( D )$. Assume that $A\in\D_2$. For $t>0$, define $v_t\in W^{s,A}( D )$ by $v_t(x)=u(tx)$. Then
                $$
                \II{D\times D}A\left(\frac{|v_t(x)-v_t(y)|}{|x-y|^s}\right)\frac{dxdy}{|x-y|^N}\leq
                \begin{cases}
                t^{s-N}\II{tD\times tD}A\left(\frac{|u(x)-u(y)|}{|x-y|^s}\right)\frac{dxdy}{|x-y|^N},& t<1,\\
                t^{sp-N}\II{tD\times tD}A\left(\frac{|u(x)-u(y)|}{|x-y|^s}\right)\frac{dxdy}{|x-y|^N},& t\geq 1,
                \end{cases}
                $$
                and also
                $$
                \II{D\times D}A\left(\frac{|v_t(x)-v_t(y)|}{|x-y|^s}\right)\frac{dxdy}{|x-y|^N}\geq
                \begin{cases}
                t^{s-N}\II{tD\times tD}A\left(\frac{|u(x)-u(y)|}{|x-y|^s}\right)\frac{dxdy}{|x-y|^N},& t\geq 1,\\
                t^{sp-N}\II{tD\times tD}A\left(\frac{|u(x)-u(y)|}{|x-y|^s}\right)\frac{dxdy}{|x-y|^N},& t<1.
                \end{cases}
                $$
                Furthermore,
                $$
                \frac{ P^1_{N,s,A}(D)}{t^s}\leq P^1_{N,s,A}(tD)\leq\frac{ P^1_{N,s,A}(D)}{t^{sp}} \ \ \mbox{if} \ t<1,
                $$
                $$
                \frac{ P^1_{N,s,A}(D)}{t^{sp}}\leq P^1_{N,s,A}(tD)\leq\frac{ P^1_{N,s,A}(D)}{t^{s}} \ \ \mbox{if} \ t\geq1.
                $$
                \item Let $ D \subseteq\RR^N$ be an open set and $t>0$. Assume that $A\in\D_2$. Then
                $$
                \frac{ P^2_{N,s,A}(D)}{t^s}\leq P^2_{N,s,A}(tD)\leq\frac{ P^2_{N,s,A}(D)}{t^{sp}} \ \ \mbox{if} \ t<1,
                $$
                $$
                \frac{ P^2_{N,s,A}(D)}{t^{sp}}\leq P^2_{N,s,A}(tD)\leq\frac{ P^2_{N,s,A}(D)}{t^{s}} \ \ \mbox{if} \ t\geq1.
                $$
                \item Let $D\subseteq\RR^N$ be  such that $BC(D)=\infty$, then  $P^1_{N,s,A}(D)=0=P^2_{N,s,A}(D)$.
            \end{enumerate}
        \end{proposition}

        \begin{proof}
            \begin{enumerate}
                \item Directly follows from the definition.\smallskip
                
                \item For the first inequality, take the change of variable $X=tx,Y=ty$ and in case of $t<1$, use \cref{A-property1} to infer $\frac{A(t^sf(x,y))}{t^sf(x,y)}<\frac{A(f(x,y))}{f(x,y)}$. In case of $t\geq1$ use \cref{lemma delta2}. For the second inequality replace $t$ by $\frac{1}{t}$.\smallskip
                
                \item Similar as case \textit{(2)}.\smallskip
                
                \item $BC(D)=\infty$ implies there exist a positive sequence $\lbrace r_n \rbrace_n$ and a sequence $\lbrace x_n \rbrace_n$ such that $r_n\to\infty$ and $B(x_n,r_n)\subseteq D$. Note that $P^1_{N,s,A}(D)\leq P^2_{N,s,A}(D)$. Then by $(1)$ and $(3)$ we infer that
                $$
                P^2_{N,s,A}(D)\leq P^2_{N,s,A}(B(x_n,r_n))=P^2_{N,s,A}(B(0,r_n))\leq\frac{ P^2_{N,s,A}(B(0,1))}{r_n^{s}}\to 0.
                $$
            \end{enumerate}	
            This finishes the proof.
        \end{proof}

        \begin{lemma}\label{BC lemma}
            Let $D\subsetneqq\RR$ be an open set with $BC(D)<\infty$ and $D=\cup_{k=1}^\infty I_k$, where $I_k$'s are disjoint intervals. Then for any $k\in\NN$, $P^1_{1,s,A}(I_k)\geq \min \lbrace BC(D)^{-sp},1 \rbrace P^1_{1,s,A}((0,1))$, where $p>1$ is as in \cref{lemma delta2}.
        \end{lemma}
        \begin{proof}
            Using (2) of \cref{monotone-dialtion-fbc}, we have
            $
            P^1_{1,s,A}(I_k)\geq \mbox{diam}(I_k)^{-\beta} P^1_{1,s,A}((0,1)),
            $
            where $$\beta=\begin{cases}
            s, &\mbox{if }\ \mbox{diam}(I_k)<1,\\
            sp, &\mbox{if }\ \mbox{diam}(I_k)\geq1.
            \end{cases}$$
            If $\mbox{diam}(I_k)<1$, $P^1_{1,s,A}(I_k)\geq P^1_{1,s,A}(0,1)$. On the other hand if $\mbox{diam}(I_k)\geq 1$, 
            $$
            P^1_{1,s,A}(I_k)\geq \mbox{diam}(I_k)^{-sp} P^1_{1,s,A}((0,1))
            \geq BC(D)^{-sp} P^1_{1,s,A}((0,1)).
            $$
            Combining the two cases we finally get the desired result.
        \end{proof}

        \bigskip

        \section{\texorpdfstring{$FOHI(s,A)$}{} and \texorpdfstring{$RFOPI(s,A)$}{} on bounded Lipschitz domains}\label{section hardy}
        In this section we shall prove \cref{Main theorem,Main theorem not true}. Our proof of \cref{Main theorem} is motivated by \cite{dyd}.	We start this section by proving some technical lemmas which will be used in the proof of \cref{Main theorem}. Let $D\subsetneqq\RR^N$ be a non-empty open set. Throughout the section we shall assume, unless stated otherwise, $\Om\subseteq D$, $A\in\D_2$ and $p$ is as in \cref{delta2 cond}.\smallskip
        
        We now fix a function $f\in \test(D)$ and define
        \begin{equation}\label{ G set}
        G=G(f,\Om;l_1,l_2):=\Bigg\lbrace x\in\Om \ \Big| \ A\left(
        \frac{|f(x)|}{\delta_x^s}\right)
        >\frac{2^{p+1}}{l_2\delta_x^N}\I{B(x,l_1\delta_x)\cap\Om}A\left(\frac{|f(x)-f(y)|}{\delta_x^s}\right)dy\Bigg\rbrace;
        \end{equation}
        where $l_1>1,\; l_2$ are positive numbers, independent of $f$, whose values are given later.
        
        \begin{lemma}\label{property1}
            Let $f,l_1,l_2$ be as above. Then	$$\I{\Om\setminus
                G}A\left(\frac{|f(x)|}{\delta_x^s}\right)dx\leq\frac{2^{p+1}l_1^{N+sp}}{l_2}\I{x\in\Om\setminus
                G}\I{y\in\Om} A\left(\frac{|f(x)-f(y)|}{|x-y|^s}\right)\frac{dydx}{|x-y|^N}.$$
        \end{lemma}
        \begin{proof}
            For $x\in\Om\setminus G$, using \cref{A-property1,lemma delta2}, we have
            \begin{multline*}
            \I{\Om} A\left(\frac{|f(x)-f(y)|}{|x-y|^s}\right)\frac{dy}{|x-y|^N}
            \geq\I{B(x,l_1\delta_x)\cap\Om}A\left(\frac{|f(x)-f(y)|}{|x-y|^s}\right)\frac{dy}{|x-y|^N}\\
            \geq\I{B(x,l_1\delta_x)\cap\Om}A\left(\frac{|f(x)-f(y)|}{(l_1\delta_x)^s}\right)\frac{dy}{(l_1\delta_x)^N}\\
            \geq\frac{1}{l_1^{N+sp}\delta_x^N}\I{B(x,l_1\delta_x)\cap\Om}A\left(\frac{|f(x)-f(y)|}{\delta_x^s}\right)dy
            \geq\frac{l_2}{l_1^{N+sp}2^{p+1}}A\left(\frac{|f(x)|}{\delta_x^s}\right).
            \end{multline*}
            Integrating over $\Om\setminus G$, we get the desired
            result.
        \end{proof}
        In the view of the above lemma, we can says that the set	$\Om\setminus G$ is a good set for fractional Orlicz Hardy inequality in the sense that the required inequality holds in it. \smallskip
        
        For any $x\in G$, let us define the set $E^*(x)=\left\lbrace y\in
        E\ | \ \frac{|f(x)|}{2}\leq|f(y)|\leq \frac{3}{2}|f(x)|\right\rbrace$. 
        \begin{lemma}\label{property2}
            Let $x\in G$ and $E\subseteq B(x,l_1\delta_x)\cap\Om$. If $\mathcal{L}^{N}(E)\geq l_2\delta_x^N$,
            then $$\frac{\mathcal{L}^{N}(E)}{2}\leq \mathcal{L}^{N}(E)-\frac{l_2\delta_x^N}{2}\leq  \mathcal{L}^{N}(E^*(x)).$$
        \end{lemma}
        \begin{proof}
            Since $x\in G$, $E\subseteq B(x,l_1\delta_x)\cap\Om$, by \cref{ G set} we have
            \begin{equation*}
            A\left(\frac{|f(x)|}{\delta_x^s}\right)
            \geq\frac{2^{p+1}}{l_2\delta_x^N}\I{B(x,l_1\delta_x)\cap\Om}A\left(\frac{|f(x)-f(y)|}{\delta_x^s}\right)dy
            \geq\frac{2^{p+1}}{l_2\delta_x^N}\I{E\setminus
                E^*(x)}A\left(\frac{|f(x)-f(y)|}{\delta_x^s}\right)dy.
            \end{equation*}
            Note that, $y\in E\setminus E^*(x)$ implies
            $|f(x)-f(y)|\geq |\ |f(x)|-|f(y)|\ |\geq\frac{|f(x)|}{2}$. Thus by using \cref{A-property1,lemma delta2} we get
            $$
            A\left(\frac{|f(x)|}{\delta_x^s}\right)\geq\frac{2}{l_2\delta_x^N}
            A\left(\frac{|f(x)|}{\delta_x^s}\right)\left\lbrace\mathcal{L}^{N}(E)-\mathcal{L}^{N}(E^*(x))\right\rbrace.
            $$
            Using the hypothesis on measure of $E$, we get $\frac{\mathcal{L}^{N}(E)}{2}\leq \mathcal{L}^{N}(E)-\frac{l_2\delta_x^N}{2}\leq  \mathcal{L}^{N}(E^*(x)).$
        \end{proof}
        \smallskip
        
        \begin{lemma}\label{property3}
            Let $E_1\subseteq\Om$ and $E_2\subseteq B(x,l_1\delta_x)\cap\Om$ be such that
            $\mathcal{L}^{N}(E_2)\geq l_2\delta_x^N$ for all $x\in E_1$. Then
            $$
            \I{E_1\cap G}A\left( \frac{|f(x)|}{\delta_x^s}\right)
            dx\leq\frac{2^{p+1}\mathcal{L}^{N}(E_1)}{\mathcal{L}^{N}(E_2)}\I{E_2}A\left(\frac{\sup\lbrace \delta_x^s \ | \ x\in E_2\rbrace \; |f(y)|}{\inf\lbrace \delta_x^s \ | \ x\in E_1\rbrace \;\delta_y^s}\right)dy.
            $$
        \end{lemma}
        \begin{proof}
            Assume that $E_1\cap G$ is non empty. First, we fix $\eta>1$ and pick $x_0\in
            E_1\cap G$ such that $\sup\limits_{x\in E_1\cap G}|f(x)|\leq\eta|f(x_0)|$. Also $|f(x_0)|\leq 2|f(y)|$ for any $y\in E_2^*(x_0)$. Hence, for any
            $y\in E_2^*(x_0)$, using \cref{A-property1,lemma delta2},
            \begin{multline*}
            \I{E_1\cap G}A\left(\frac{|f(x)|}{\delta_x^s}\right)dx
            \leq \mathcal{L}^{N}(E_1\cap G)A\left(\frac{\sup\limits_{x\in E_1\cap
                    G}|f(x)|}{\inf\lbrace \delta_x^s \ | \ x\in E_1\cap G\rbrace }\right)\\
            \leq \mathcal{L}^{N}(E_1\cap G)A\left(\frac{2\eta|f(y)|}{\inf\lbrace \delta_x^s \ | \ x\in E_1\cap G\rbrace}\right)
            \leq 2^p\eta^p\mathcal{L}^{N}(E_1)A\left(\frac{\sup\lbrace \delta_y^s \ | \ y\in E_2^*(x_0)\rbrace \;  |f(y)|}{\inf\lbrace \delta_x^s\ | \ x\in E_1\rbrace \;  \delta_y^s}\right).
            \end{multline*}
            Integrating over $y\in E_2^*(x_0)$, using \cref{property2} and the fact that $E_2^*(x_0)\subseteq E_2$, we obtain
            
            $$
            \I{E_1\cap G}A\left( \frac{|f(x)|}{\delta_x^s}\right)
            dx\leq\frac{2^{p+1}\eta^p \mathcal{L}^{N}(E_1)}{\mathcal{L}^{N}(E_2)}\I{E_2}A\left(\frac{\sup\lbrace \delta_x^s \ | \ x\in E_2\rbrace \; |f(y)|}{\inf\lbrace \delta_x^s\ | \ x\in E_1\rbrace \;  \delta_y^s}\right)dy.
            $$
            Letting $\eta\to 1$ the proof follows.
        \end{proof}
    \begin{lemma}\label{iterative process}
        Let $\Om\subseteq D \subseteq \RR^N$ be two open sets, $f\in\test(D)$, $l_1,l_2>0$, $G=G(f,\Om;l_1,l_2), 0<\gamma<1$ and $m\in\NN$. Assume $\Om=\cup_{j=0}^\infty A_j$, $\mathcal{L}^N(A_i\cap A_j)=0 \ \forall\ j\neq i$ and $\exists\; n_0\in\NN$ such that $f\equiv 0$ on $A_j$ for $j\geq n_0$.
        \begin{enumerate}
            \item If $\ \forall \ j\in\NN$,
            $$
            \I{G\cap A_j}A\left(\frac{|f(x)|}{\delta_x^s}\right)dx\leq \gamma\I{A_{j+m}}A\left(\frac{|f(x)|}{\delta_x^s}\right)dx,
            $$
            then there exists a constant $c>0$ such that
            $$
            \I{\Om}A\left(\frac{|f(x)|}{\delta_x^s}\right)dx\leq c\II{\Om\times\Om}A\left(\frac{|f(x)-f(y)|}{|x-y|^s}\right)\frac{dxdy}{|x-y|^N}.
            $$
            
            \item  If $\ \forall \ j\geq m$,
            $$
            \I{G\cap A_j}A\left(\frac{|f(x)|}{\delta_x^s}\right)dx\leq \gamma\I{A_{j-m}}A\left(\frac{|f(x)|}{\delta_x^s}\right)dx,
            $$
            then there exists a constant $c>0$ such that
            $$
            \I{\Om}A\left(\frac{|f(x)|}{\delta_x^s}\right)dx\leq c\II{\Om\times\Om}A\left(\frac{|f(x)-f(y)|}{|x-y|^s}\right)\frac{dxdy}{|x-y|^N}+c\I{A_0\cup\cdots\cup A_{m-1}}A\left(\frac{|f(x)|}{\delta_x^s}\right)dx.
            $$
        \end{enumerate}
    \end{lemma}
\begin{proof}
    (1)
    \begin{multline*}
    \I{\Om\cap
        G}A\left(\frac{|f(x)|}{\delta_x^s}\right)dx
    =\sum_{j=0}^{\infty}\I{A_j\cap
        G}A\left(\frac{|f(x)|}{\delta_x^s}\right)dx
    \leq \sum_{j=0}^{\infty}\gamma\I{A_{j+m}}A\left(\frac{|f(x)|}{\delta_x^s}\right)dx\\
    =\sum_{j=0}^{\infty}\gamma\left[\I{A_{j+m}\cap G}A\left(\frac{|f(x)|}{\delta_x^s}\right)dx+\I{A_{j+m}\setminus G}A\left(\frac{|f(x)|}{\delta_x^s}\right)dx\right]\\
    \leq \sum_{j=0}^{\infty}\gamma\left[\gamma\I{A_{j+2m}}A\left(\frac{|f(x)|}{\delta_x^s}\right)dx+\I{A_{j+m}\setminus G}A\left(\frac{|f(x)|}{\delta_x^s}\right)dx\right]\\
    \leq\cdots
    \leq \sum_{j=0}^{\infty}\sum_{k=1}^{\infty}\gamma^{k}\I{A_{j+km}\setminus G}A\left(\frac{|f(x)|}{\delta_x^s}\right)dx.
    \end{multline*}
    Here we used the fact that $f\equiv 0$ for large $j$ forcing the iterative process to terminate. So we have
    \begin{equation*}
    \I{\Om\cap
        G}A\left(\frac{|f(x)|}{\delta_x^s}\right)dx
    \leq \sum_{k=1}^{\infty}\gamma^k\I{\Om\setminus G}A\left(\frac{|f(x)|}{\delta_x^s}\right)dx.
    \end{equation*}
    Hence by using above estimate and \cref{property1} we get
    \begin{multline*}
    \I{\Om} A\left(\frac{|f(x)|}{\delta_x^s}\right)dx
    =\I{\Om\cap G} A\left(\frac{|f(x)|}{\delta_x^s}\right)dx+\I{\Om\setminus G} A\left(\frac{|f(x)|}{\delta_x^s}\right)dx\\
    \leq\sum_{k=0}^\infty\gamma^k\I{\Om\setminus
        G}A\left(\frac{|f(x)|}{\delta_x^s}\right)dx
    \leq 
    c\I{\Om}\I{\Om}
    A\left(\frac{|f(x)-f(y)|}{|x-y|^s}\right)\frac{dxdy}{|x-y|^N}.
    \end{multline*}
    This completes the proof of (1).\\
    
    (2) Let for each $j\geq 0$, $k_j$ be the largest nonnegative integer such that $j-k_jm \geq 0$, so that $0\leq j-k_jm\leq m-1$. Then $k_j$-many repeated applications of the inequality in the hypothesis gives
    \begin{multline*}
        \I{A_j\cap G}A\left( \frac{|f(x)|}{\delta_x^s}\right)
        dx
        \leq  \gamma\I{A_{j-m}}A\left(\frac{|f(y)|}{\delta_y^s}\right)dy\\
        = \gamma \left[ \I{A_{j-m}\setminus G}A\left(\frac{|f(y)|}{\delta_y^s}\right)dy+\I{A_{j-m}\cap G }A\left(\frac{|f(y)|}{\delta_y^s}\right)dy\right]\\
        \leq\cdots \leq \gamma^{k_j}\I{A_{j-k_jm}\cap G }A\left(\frac{|f(y)|}{\delta_y^s}\right)dy\ +\ \sum_{k=1}^{k_j}\gamma^k\I{A_{j-km}\setminus G }A\left(\frac{|f(y)|}{\delta_y^s}\right)dy\\
        \leq \gamma^{k_j}\I{A_0\cup\cdots\cup A_{m-1}} A\left(\frac{|f(y)|}{\delta_y^s}\right)dy\ +\ \sum_{k=1}^{k_j}\gamma^k\I{A_{j-km}\setminus G }A\left(\frac{|f(y)|}{\delta_y^s}\right)dy.
    \end{multline*}
    Summing over $j$,
    
    \begin{equation*}
    \I{\Om\cap G}A\left( \frac{|f(x)|}{\delta_x^s}\right)
    dx
    \leq\ \frac{\gamma}{1-\gamma}\left[\ \I{A_0\cup\cdots\cup A_{m-1}}A\left( \frac{|f(x)|}{\delta_x^s}\right)
    dx+\ \I{\Om\setminus G}A\left( \frac{|f(x)|}{\delta_x^s}\right)
    dx\right].
    \end{equation*}
    Proceeding as in part (1), the lemma follows.
    
\end{proof}
        Before proving \cref{Main theorem}, we need a geometric decomposition of a bounded Lipschitz domain $D$, given in \cite[p.~581]{dyd}. We outline the construction for the sake of completeness.\smallskip
        
        Let us denote $x=(x_1,x_2,\cdots,x_{N-1},x_N)=(\tilde{x},x_N)\in\RR^N$ with
        $\tilde{x}\in\RR^{N-1}$, $x_N\in\RR$. $D\subsetneqq\RR^N$ shall be assumed to be a bounded Lipschitz domain throughout the rest of the section. For any $z\in\partial D$ there are linear isometry $L_z:\RR^N\to\RR^N$ and Lipschitz
        function $\phi_z:\RR^{N-1}\to\RR$ such that $\mbox{Lip}(\phi_z)\leq \mbox{Lip}(D)$ and
        $$
        L_z(D)\cap B(L_z(z),r_0)=\ \lbrace x\in\RR^N\ | \ x_N>\phi_z(\tilde{x})\rbrace \cap B(L_z(z),r_0),
        $$
        for some positive $r_0$ which depends only on $D$(existence of $r_0$ is guaranteed because $\partial D$ is compact). Without loss of generality, we can assume $L_z$ to be the identity map. Because otherwise we can work with the Lipschitz domain $L_z(D)$ and the point $L_z(z)$, then pull back the construction to the original domain $D$ and the point $z$ via $L_z^{-1}$. For $x\in\RR^N$ we set
        $$
        V_z(x):=|x_N-\phi_z(\tilde{x})|.
        $$ 
        For $E\subseteq\RR^{N-1}$ and $r>0$ define
        $$
        Q_z(E,r):=\lbrace x\in D\ | \ \tilde{x}\in E,0<V_z(x)\leq r\rbrace. 
        $$
        Set
        $$
        K_r:=\left\{ x\in\RR^{N-1} \ \big| \ |x_\ell-z_\ell|\leq \frac{r}{2},\;\ell=1,2,\cdots,N-1\right\}.
        $$
        That is $K_r$ is an $N-1$ dimensional square of side-length $r$.
        $Q_z(K_\rho,r)$ is referred to as a Lipschitz box and is denoted by $Q_z(r)$.
        We now choose $\rho>0$ small enough such that $Q_z(\rho)\subseteq D\cap B(z,r_0/2)$.
        Let $x\in Q_z(\rho)$. Then using the fact that $\phi_z(w)=w_N$ for any $w\in\partial D\cap B(z,r_0)$,
        $$
        V_z(x)=|x_N-\phi_z(\tilde{x})|\leq |x_N-w_N|+|\phi_z(\tilde{w})-\phi_z(\tilde{x})|\leq |x_N-w_N|+\mbox{Lip}(D) |\tilde{w}-\tilde{x}|
        \leq (1+\mbox{Lip}(D))|w-x|.
        $$
        Since $w$ is arbitrary, we get
        \begin{equation}\label{dist est1}
        \frac{V_z(x)}{(1+\mbox{Lip}(D))}\leq\delta_x\leq V_z(x),\quad \forall \ x\in Q_z(\rho).
        \end{equation}
        For $j\in\NN\cup\lbrace0\rbrace$ consider the dyadic decomposition of $K_\rho$ (by dissecting sides first and then decomposing the cube into smaller cubes and proceeding by induction) into the union of $(N-1)$-dimensional cubes $K^i_j$ for $1\leq i\leq2^{j(N-1)}$. Now for each $j$, $K^i_j$'s have
        disjoint interiors and have sides-length $\frac{\rho}{2^j}$. Set $Q^i_j:=Q_z(K^i_j,\rho)$ and define
        \begin{equation}\label{Ak set}
        A_k=Q_z(K_\rho,\rho/2^k)\setminus Q_z(K_\rho,\rho/2^{k+1}), \text{ for
        }k\in\NN\cup\lbrace0\rbrace.
        \end{equation}
        Then $A_k$'s are mutually disjoint sets with $\cup_{k=0}^\infty
        A_k=Q_z(\rho)$. Thus the Lipschitz box $Q_z(\rho)$ is the union of sets $A_j\cap Q^i_j$, $i,j\in\NN$ whose pairwise intersections are measure zero sets. Moreover using \cref{Ak set,dist est1}, for $k\geq j$, we have
        \begin{equation}\label{dist est2}
        \frac{\rho}{2^{k+1}(1+\mbox{Lip}(D))}\leq\delta_x\leq\frac{\rho}{2^k},
        \text{ for }x\in A_k\cap Q^i_j
        \end{equation}
        and
        \begin{equation}\label{measure}
        \mathcal{L}^{N}(A_k\cap Q^i_j)=\left(\frac{\rho}{2^k}-\frac{\rho}{2^{k+1}}\right)\left(\frac{\rho}{2^j}\right)^{N-1}=\frac{\rho^N}{2^{k+1+jN-j}}.
        \end{equation}
        Let $x\in A_k\cap Q^i_j$ and $y\in A_j\cap Q^i_j$ where $k\geq j$. Then $\frac{\rho}{2^{k+1}} \leq |x_N-\phi_z(\tilde{x})| \leq \frac{\rho}{2^{k}}$ and $\frac{\rho}{2^{j+1}} \leq |y_N-\phi_z(\tilde{y})| \leq \frac{\rho}{2^{j}}$. Moreover both $x$ and $y$ lie `above' the same dyadic cube $Q^i_j$ which has side length $\frac{\rho}{2^j}$. So for $1\leq \ell \leq N-1$, $|x_\ell-y_\ell|\leq\frac{\rho}{2^{j}}$. Therefore $|\tilde{x}-\tilde{y}|\leq\sqrt{N-1}\frac{\rho}{2^{j}}$. Again, using these
        \begin{multline*}
        |x_N-y_N|=|x_N-\phi_z(\tilde{x})+\phi_z(\tilde{x})-\phi_z(\tilde{y})+\phi_z(\tilde{y})-y_N|\leq \frac{\rho}{2^k}+\mbox{Lip}(D)|\tilde{x}-\tilde{y}|+\frac{\rho}{2^j}\\
        \leq \frac{\rho}{2^{j-1}}+\mbox{Lip}(D)\sqrt{N-1}\frac{\rho}{2^{j}}.
        \end{multline*}
        So we finally get
        \begin{multline*}
        |x-y|^2\leq \left(\sqrt{N-1}\frac{\rho}{2^{j}}\right)^2+ \left(\frac{\rho}{2^{j-1}}+\mbox{Lip}(D)\sqrt{N-1}\frac{\rho}{2^{j}}\right)^2\\
        =\left(\frac{\rho}{2^{j}}\right)^2\left(N-1+ \left(2+\mbox{Lip}(D)\sqrt{N-1}\right)^2\right).
        \end{multline*}
        This gives us for $k\geq j$,
        \begin{equation}\label{dist btwn pts}
        |x-y|\leq\frac{\rho}{2^{j}}\sqrt{N-1+(\mbox{Lip}(D)\sqrt{N-1}+2)^2}\quad\mbox{ for }x\in A_k\cap Q^i_j\mbox{ and }y\in A_j\cap Q^i_j.
        \end{equation}
        
        \begin{figure}[H]
            \includegraphics[height=7cm]{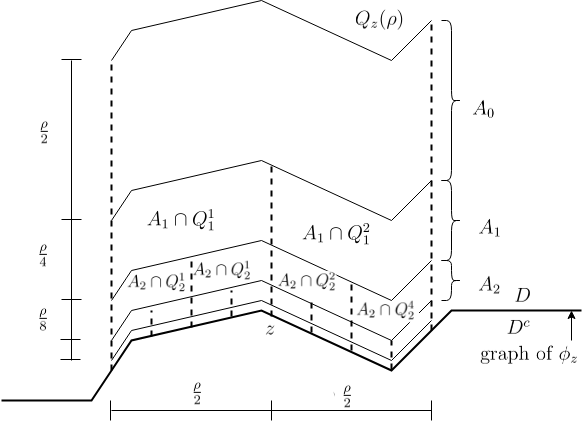}
            \caption{Lipschitz box $Q_z(\rho),\;N=2.$}
            \label{figure}
        \end{figure}
        For the next result we consider $Q=Q_z(\rho)$ defined above.
        \begin{lemma}\label{Hardy in cube}
            Let $\lim\limits_{\lambda\to 0+}\alpha_{s,A}(\lambda)=0$, where $\alpha_{s,A}$
            is as in \cref{alpha fnc}. Then there exists a constant $c=c(D,N,A,s)>0$ such that for all $f\in \test(D)$ we have
            $$
            \I{Q_z(\rho)} A\left(\frac{|f(x)|}{\delta_x^s}\right)dx\leq c\II{Q_z(\rho)\times Q_z(\rho)}
            A\left(\frac{|f(x)-f(y)|}{|x-y|^s}\right)\frac{dxdy}{|x-y|^N}.
            $$
        \end{lemma}
        \begin{proof}
            From the hypothesis on $s,A$ it follows that there exists $m\in\mathbb N$ such that 
            $$
            2^{2+p}(1+\mbox{Lip}(D))\alpha_{s,A}(2^{s-ms}(1+\mbox{Lip}(D))^{s})<\frac{1}{2}.$$ Set 
            
            $$
            l_1=2(1+\mbox{Lip}(D))\sqrt{N-1+(\mbox{Lip}(D)\sqrt{N-1}+2)^2}
            \mbox{ and }l_2=\frac{1}{2^{m+1}}.
            $$
            Let $f\in \test(D)$ and $G=G(f,Q_z(\rho);l_1,l_2)$. Set $E_1:=A_j\cap Q^i_j$ and $E_2:=A_{j+m}\cap Q^i_j$. For these choices of $l_1$ and $l_2$, from \cref{measure,dist est2}, we get $\mathcal{L}^{N}(E_2)=\frac{ \rho^N}{2^{1+m+jN}}$ and  $\delta_x\leq\frac{
                \rho}{2^{j+m}}\text{ for }x\in E_2$. Also utilizing \cref{dist btwn pts,dist est2}, $E_2\subseteq B(x,l_1\delta_x)$. Thus the
            sets $E_1$ and $E_2$ satisfy the hypotheses of \cref{property3}. Thus we obtain, using \cref{property3}, \cref{dist est2} and the choice of $m$,
            \begin{multline*}
            \I{(A_j\cap Q^i_j)\cap G} A\left(\frac{|f(x)|}{\delta_x^s}\right)
            dx
            \leq 2^{1+p+m}\I{A_{j+m}\cap
                Q^i_j}A\left(2^{s-ms}(1+\mbox{Lip}(D))^{s}\frac{|f(y)|}{\delta_y^s}\right)dy\\
            =2^{2+p}(1+\mbox{Lip}(D))\I{A_{j+m}\cap
                Q^i_j}\frac{A\left(2^{s-ms}(1+\mbox{Lip}(D))^{s}\frac{|f(y)|}{\delta_y^s}\right)}{2^{1-m}(1+\mbox{Lip}(D))A\left(\frac{|f(y)|}{\delta_y^s}\right)}A\left(\frac{|f(y)|}{\delta_y^s}\right)dy\\
            \leq 2^{2+p}(1+\mbox{Lip}(D))\alpha_{s,A}(2^{s-ms}(1+\mbox{Lip}(D))^{s})\I{A_{j+m}\cap
                Q^i_j}A\left(\frac{|f(y)|}{\delta_y^s}\right)dy
            \leq\frac{1}{2}\I{A_{j+m}\cap Q^i_j}A\left(\frac{|f(y)|}{\delta_y^s}\right)dy.
            \end{multline*}
            Summing over $i=1,\cdots,2^{j(N-1)}$ we obtain
            
            \begin{equation}\label{iteration}
            \I{A_j\cap
                G}A\left(\frac{|f(x)|}{\delta_x^s}\right)dx
            \leq \frac{1}{2}\I{A_{j+m}}A\left(\frac{|f(x)|}{\delta_x^s}\right)dx.
            \end{equation}
           The lemma follows from \cref{iterative process}.
        \end{proof}
        \begin{proof}[\textbf{Proof of \Cref{Main theorem}}] Consider the
            following two sets
            
            $$
            D_1=\lbrace x\in D \ | \ \delta_x\geq\tilde{\rho}\rbrace ,
            \text{ and }
            D_2=\lbrace x\in D \ | \ \delta_x<\tilde{\rho}\rbrace ,
            $$
            where $\tilde{\rho}>0$ is sufficiently small. Then by \cref{Hardy in
                cube} and compactness of $\partial D$, we have
            \begin{equation}\label{bdd lip estmt 1}
            \I{D_2} A\left(\frac{|f(x)|}{\delta_x^s}\right)dx\leq\sigma
            c\I{D_2}\I{D_2}
            A\left(\frac{|f(x)-f(y)|}{|x-y|^s}\right)\frac{dxdy}{|x-y|^N},
            \end{equation}
            since $D_2$ may be covered by sets of the form $Q_{z_k}(\rho)$ such that
            every $x\in D_2$ belongs to at-most $\sigma=\sigma(D,N)\in\NN$ sets of type
            $Q_{z_k}(\rho)$. This is possible for sufficiently small
            $\tilde{\rho}<\rho$, and such a $\tilde{\rho}$ may be chosen to depend
            only
            on $\mbox{Lip}(D),\;N,\;r_0$. We now take $G=G(f,D;l_1,l_2)$, where $l_1=\frac{\mbox{diam}(D)}{\rho}$ and $l_2=\frac{\mathcal{L}^{N}(D_2)}{\mbox{diam}(D)^N}$. Set $E_1=D_1$ and $E_2=D_2$. Then for any $x\in E_1$,
            $$
            l_1\delta_x=\mbox{diam}(D)\frac{\delta_x}{\rho}\geq \mbox{diam}(D),
            $$
            which implies $E_2\subseteq B(x,l_1\delta_x)$ and 
            $$
            \mathcal{L}^N(E_2)=\mbox{diam}(D)^Nl_2\geq \delta_x^Nl_2.
            $$
            We can now apply \cref{property3}
            to get 
            \begin{equation}\label{bdd lip estmt 2}
            \I{D_1\cap G}A\left(\frac{|f(x)|}{\delta_x^s}\right)dx\leq
            c(E_1,E_2)\I{D_2}A\left(\frac{|f(x)|}{\delta_x^s}\right)dx,
            \end{equation}
            and also from \cref{property1} we get
            \begin{equation}\label{bdd lip estmt 3}
            \I{D_1\setminus G}A\left(\frac{|f(x)|}{\delta_x^s}\right)dx\leq
            c\I{D_1\setminus G}\I{D}
            A\left(\frac{|f(x)-f(y)|}{|x-y|^s}\right)\frac{dydx}{|x-y|^N}.
            \end{equation}
            Combining the estimates \cref{bdd lip estmt 1}, \cref{bdd lip estmt 2}
            and \cref{bdd lip estmt 3} we conclude $H_{N,s,A}(D)>0$.
            If $A\in\D_2$, then using \cref{hardy to poincare}, $P^1_{N,s,A}(D)>0$. This completes the proof.
        \end{proof}
            
            \begin{proof}[\textbf{Proof of \Cref{Main theorem not true}}]
                For any $\varepsilon\in(0,1)$ denote $D_\varepsilon:=\lbrace x\in D \ | \
                \mbox{dist}(x,\partial D)<\varepsilon\rbrace $. Let $f_\varepsilon\in \test(D)$
                be such that $f_\varepsilon\equiv1$ on $D\setminus D_\varepsilon$,
                $0\leq f_\varepsilon\leq 1$ on $D$ and $|\nabla
                f_\varepsilon|<\frac{c_1}{\varepsilon}$ on $D$ for some constant $c_1>0$.
                Now
                \begin{multline*}
                \II{D \times D}A\left(\frac{| f_\varepsilon(x)-f_\varepsilon(y) |}{| x-y|^s}\right)\frac{dxdy}{|x-y|^N}=\II{D_\varepsilon \times
                    D_\varepsilon}A\left(\frac{| f_\varepsilon(x)-f_\varepsilon(y) |}{| x-y
                    |^s}\right)\frac{dxdy}{|x-y|^N}\\
                + 2\II{D_\varepsilon^c \times D_\varepsilon}A\left(\frac{|
                    f_\varepsilon(x)-f_\varepsilon(y) |}{| x-y |^s}\right)\frac{dxdy}{|x-y|^N}
                \leq 3\II{D \times D_\varepsilon}A\left(\frac{|
                    f_\varepsilon(x)-f_\varepsilon(y) |}{|x-y|^s}\right)\frac{dxdy}{|x-y|^N}\\
                \leq 3 \I{D_\varepsilon}\Bigg\lbrace  \I{y\in D, |x-y|<\varepsilon}A\left(
                \frac{| f_\varepsilon(x)-f_\varepsilon(y) |}{| x-y
                    |^s}\right)\frac{dy}{|x-y|^N}
                \\+ \I{y\in D, |x-y|>\varepsilon}A\left( \frac{|
                    f_\varepsilon(x)-f_\varepsilon(y) |}{| x-y
                    |^s}\right)\frac{dy}{|x-y|^N}\Bigg\rbrace  dx
                =3 \I{D_\varepsilon}(I_1(x)+I_2(x))dx.
                \end{multline*}
                Using $|\nabla f_\varepsilon|<\frac{c_1}{\varepsilon}$ on $D$, we get
                \begin{multline*}
                I_1(x)=\I{y\in D,|x-y|<\varepsilon}A\left(\frac{|
                    f_\varepsilon(x)-f_\varepsilon(y) |}{| x-y
                    |^s}\right)\frac{dy}{|x-y|^N}\leq \I{|x-y|<\varepsilon}A\left(\frac{c_1|
                    x-y|^{1-s}}{\varepsilon}\right)\frac{dy}{|x-y|^N}\\
                =\I{|y|<\varepsilon}A\left(\frac{c_1|y|^{1-s}}{\varepsilon}\right)\frac{dy}{|y|^N}
                \leq c\I{0}^{\varepsilon}A\left(\frac{c_1r^{1-s}}{\varepsilon}\right)\frac{dr}{r} =c\I{0}^{c_1\varepsilon^{-s}} \frac{A(z)}{ z}dz.
                \end{multline*}
                The fact that $0\leq f_\varepsilon(x)\leq1$ on $D$ gives
                \begin{multline*}
                I_2(x)=\I{y\in D, |x-y|>\varepsilon}A\left(\frac{|
                    f_\varepsilon(x)-f_\varepsilon(y) |}{| x-y
                    |^s}\right)\frac{dy}{|x-y|^N}\leq \I{|x-y|>\varepsilon}A(2|
                x-y|^{-s})\frac{dy}{|x-y|^N}\\
                = c\I{\varepsilon}^\infty \frac{A(2r^{-s})}{ r}dr\leq c
                \I0^{2\varepsilon^{-s}}\frac{A(z)}{z}dz.
                \end{multline*}
                The hypothesis on $D$ that it has bounded Lipschitz boundary implies that $\mathcal{L}^N(D_\varepsilon)$ is bounded above as well as bounded bellow by a constant multiple of $\varepsilon$. Let $\frac{1}{c_2}:=\max\lbrace c_1,2\rbrace$. Then we have, as $\varepsilon\to0$,
                \begin{equation}\label{th2-1}
                \II{D \times D}A\left(\frac{| f_\varepsilon(x)-f_\varepsilon(y) |}{| x-y
                    |^s}\right)\frac{dxdy}{|x-y|^N}\leq c\mathcal{L}^N(D_\varepsilon)
                \I0^{(c_2\varepsilon)^{-s}}\frac{A(z)}{z}dz
                \leq c\ (c_2\varepsilon)
                \I0^{(c_2\varepsilon)^{-s}}\frac{A(z)}{z}dz\to\beta.
                \end{equation}
                Now $f_\varepsilon \in \test(D)$, $f_\varepsilon\to 1$ pointwise a.e. in $D$ and $A$ is continuous. Therefore by Fatou's lemma,
                \begin{equation}\label{th2-2}
                \lim\limits_{\varepsilon\to 0}\I{D}A\left(\frac{| f_\varepsilon(x)|}{\delta_x^s}\right)dx \geq \I{D}A\left(\frac{1}{\delta_x^s}\right)dx>0
                \end{equation}
                and
                \begin{equation}\label{th2-3}
                \lim\limits_{\varepsilon\to0}\I{D}A(| f_\varepsilon(x)|)dx \geq A(1)\mathcal{L}^{N}(D)>0.
                \end{equation}
                
                Now to prove $(1)$ assume $\beta=0$. Then LHS of \cref{th2-1}, which is also the numerator in the definition of both $H_{N,s,A}$ and $P^1_{N,s,A}$, converges to $0$. We now use \cref{th2-2,th2-3} to conclude $H_{N,s,A}(D)=0$ and $P^1_{N,s,A}(D)=0$ respectively. This proves (1).
                
                Now we prove (2). Applying L'hospital rule  to \cref{integralcond}, we get $\lim\limits_{\varepsilon\to0}\varepsilon A(\varepsilon^{-s})=\beta$. So there exists $n_0\in\NN$ such that for $n\geq n_0$, $A(n^s)\geq\frac{\beta n}{2}$. Using \cref{th2-2} we have
                \begin{multline*}
                \lim\limits_{ \frac{1}{n}\to 0}\I{D}A\left(
                \frac{|f_{\frac{1}{n}}(x)|}{\delta_x^s}\right)
                dx\geq\I{D}A\left(  \frac{1}{\delta_x^s}\right)
                dx
                \geq \sum_{n=n_0}^{\infty} \I{D_\frac{1}{n}\setminus D_\frac{1}{n+1}}A\left(
                \frac{1}{\delta_x^s}\right)
                dx\\
                \geq \sum_{n=n_0}^{\infty} \I{D_\frac{1}{n}\setminus D_\frac{1}{n+1}}A\left(n^{s}\right)
                dx = 
                \sum_{n=n_0}^{\infty}\frac{\beta n}{2}\mathcal{L}^N\left(D_\frac{1}{n}\setminus D_\frac{1}{n+1}\right) = \sum_{n=n_0}^{\infty}\frac{\beta n}{2}\left(\frac{1}{n}-\frac{1}{n+1}\right) =\infty.
                \end{multline*}
                The proof follows after observing that $\beta > 0$ in \cref{th2-1}.
            \end{proof}
            The conclusions we can draw as an application of \cref{Main theorem,Main theorem not true}, for any bounded Lipschitz domain $ D $ and for any $q>1$ are shown bellow in \cref{table}.
            \begin{table}[H]
                \caption{Conclusions we can draw from \cref{Main theorem,Main theorem not true}}
                \label{table}
                \centering
                \begin{tabular}{|c |r|r|r|r|}
                    \hline\hline
                    $A(t)$ & $H_{N,s,A}(D)>0$ & $H_{N,s,A}(D)=0$ & $P^1_{N,s,A}(D)>0$ & $P^1_{N,s,A}(D)=0$\\ [0.5ex]   \hline\hline          
                    $t^q$  &$s\in(\frac{1}{q},1)$&$s\in(0,\frac{1}{q}]$&$s\in(\frac{1}{q},1)$&$s\in(0,\frac{1}{q})$\\ [1ex]
                    \hline
                    $t^q(1+|\log{t}|)$ &$s\in(\frac{1}{q},1)$&$s\in(0,\frac{1}{q})$&$s\in(\frac{1}{q},1)$&$s\in(0,\frac{1}{q})$\\[1ex]
                    \hline
                    $\frac{t^q}{\log{(e+t)}}$&$s\in(\frac{1}{q},1)$&$s\in(0,\frac{1}{q}]$&$s\in(\frac{1}{q},1)$&$s\in(0,\frac{1}{q}]$\\[1ex]
                    \hline
                    $(1+t)\log{(1+t)}-t$&NA&$s\in(0,1)$&NA&$s\in(0,1)$.\\[1ex] \hline
            \end{tabular}\end{table}
        \bigskip
        
    \section{Proof of \Cref{Main Theorem Point Hardy,Main Theorem Other Hardy}}\label{section other hardy}
    In this section we prove \cref{Main Theorem Point Hardy,Main Theorem Other Hardy}. The notations and conventions followed in this section will be the same as that in \cref{section hardy}. We need \cref{hardy lambda infty,hardy lambda zero} to prove \cref{Main Theorem Point Hardy} and to prove \cref{Main Theorem Other Hardy} we further need \cref{Hardy in cube unbounded}.\smallskip
    
        \begin{lemma}\label{hardy lambda infty}
        Suppose $\liminf\limits_{\lambda\to\infty}\lambda^{\frac{1-N}{s}}\alpha_{s,A}(\lambda)=0$. Then there exists a constant $c=c(s,A,N)>0$ such that for any $r>0$ and any $f\in\test(\RR^N)$,
        $$
        \I{B(0,r)^c}A\left(\frac{|f(x)|}{|x|^s}\right)dx\leq
        c\II{B(0,r)^c\times B(0,r)^c}
        A\left(\frac{|f(x)-f(y)|}{|x-y|^s}\right)\frac{dydx}{|x-y|^N}.
        $$
    \end{lemma}
\begin{proof}
    We can assume $f\in\test(\RR^N\setminus \lbrace0\rbrace)$, because for the sake of proving the above inequality the value of $f$ inside $B(0,r)$ does not matter. Let $D:=\RR^N\setminus \lbrace0\rbrace$ so that $\delta_x=|x|$, $\Om:=B(0,r)^c$, $l_1=2^{m+1}$, $l_2=\mathcal{L}^N(B(0,1))2^{mN}(2^N-1)$, $G:=G(f,\Om;l_1,l_2)$ and $A_k=B(0,2^{k+1}r)\setminus B(0,2^{k}r)$ for $k\in\NN\cup\lbrace0\rbrace$. For a fixed $j\in\NN\cup\lbrace 0 \rbrace$, set $E_1=A_j$ and $E_2=A_{j+m}$. Then for $x\in E_1$, $2^{j}r\leq \delta_x\leq 2^{j+1}r$, which gives $\mbox{dist}(x,A_{j+m})\leq 2^{j+m+1}r \leq l_1\delta_x$. Again $\mathcal{L}^N(A_k)=\mathcal{L}^N(B(0,1))2^{kN}r^N(2^N-1)$. So $E_1,E_2$ satisfy the hypotheses of \cref{property3}. We can then conclude
    \begin{multline*}
    \I{A_j\cap G}A\left( \frac{|f(x)|}{\delta_x^s}\right)
    dx\leq\frac{2^{p+1}\mathcal{L}^{N}(A_j)}{\mathcal{L}^{N}(A_{j+m})}\I{A_{j+m}}A\left(\frac{\sup\lbrace \delta_x^s \ | \ x\in A_{j+m}\rbrace \; |f(y)|}{\inf\lbrace \delta_x^s \ | \ x\in A_{j}\rbrace \;\delta_y^s}\right)dy\\
    =2^{p+1-mN}\I{A_{j+m}}A\left(\frac{(2^{j+m+1}r)^s \; |f(y)|}{(2^{j}r)^s \;\delta_y^s}\right)dy
    =2^{p+1-mN}\I{A_{j+m}}A\left(\frac{2^{(m+1)s} \; |f(y)|}{\delta_y^s}\right)dy\\
    =2^{p+1+N}\I{A_{j+m}}\frac{A\left(\frac{2^{(m+1)s} \; |f(y)|}{\delta_y^s}\right)}{2^{(m+1)N}A\left(\frac{|f(y)|}{\delta_y^s}\right)}A\left(\frac{|f(y)|}{\delta_y^s}\right)dy\\
    =2^{p+1+N}2^{(m+1)(1-N)}\alpha_{s,A}(2^{(m+1)s})\I{A_{j+m}}A\left(\frac{|f(y)|}{\delta_y^s}\right)dy.
    \end{multline*}
    Now we use the hypothesis and choose large enough $m=m(s,A,N)\in\NN$, so that
    $$
    \I{A_j\cap G}A\left( \frac{|f(x)|}{\delta_x^s}\right)
    dx\leq\frac{1}{2}\I{A_{j+m}}A\left(\frac{|f(y)|}{\delta_y^s}\right)dy.
    $$
    Now we can apply \cref{iterative process} to complete the proof.
    \end{proof}

\begin{lemma}\label{hardy lambda zero}
Suppose $\liminf\limits_{\lambda\to0+}\lambda^{\frac{1-N}{s}}\alpha_{s,A}(\lambda)=0$. Then there exists a constant $c=c(s,A,N)>0$ and $m=m(s,A,N)\in\NN$ such that for any $r>0$ and any $f\in\test(\RR^N)$,
\begin{multline*}
\I{B(0,2^mr)^c}A\left(\frac{|f(x)|}{\delta_x^s}\right)dx\\
\leq
c\left[\ \II{B(0,2^mr)^c\times B(0,2^mr)^c}
A\left(\frac{|f(x)-f(y)|}{|x-y|^s}\right)\frac{dydx}{|x-y|^N}\ +\I{B(0,2^mr)\setminus B(0,r)}A\left(\frac{|f(x)|}{\delta_x^s}\right)dx\right].
\end{multline*}
\end{lemma}

\begin{proof}
    From the hypothesis on $s$ and $A$, there exists $m\in\NN$ such that
    $$
    2^{p+1+N}2^{(1-m)(1-N)}\alpha_{s,A}\left(2^{(1-m)s}\right)<\frac{1}{2}.  
    $$
    We fix this $m$. Let $D=\test(\RR^N\setminus\lbrace 0 \rbrace)$, $\Om=B(0,r)^c$, $l_1=2$, $l_2=\mathcal{L}^N(B(0,1))2^{-mN}(2^N-1)$, $G=G(f,\Om;l_1,l_2)$ and $A_k=B(0,2^{k+1}r)\setminus B(0,2^{k}r)$, then $\mathcal{L}^N(A_k)=\mathcal{L}^N(B(0,1))2^{kN}r^N(2^N-1)$. As in the previous lemma, we assume $f\in\test(D)$. For $j\geq m$, set $E_1=A_j$ and $E_2=A_{j-m}$. So for $x\in A_j$, $\mbox{dist}(x,A_{j-m})\leq 2^{j+1}r\leq l_1\delta_x$. Also $\frac{\mathcal{L}^N(A_{j-m})}{\delta_x^N}\geq\mathcal{L}^N(B(0,1))2^{-mN}(2^N-1)=l_2$. Hence, for any $j\geq m$, we can apply \cref{property3} to get 
    \begin{multline*}
        \I{A_j\cap G}A\left( \frac{|f(x)|}{|x|^s}\right)
        dx
        \leq 2^{p+1+mN}\I{A_{j-m}}A\left(\frac{2^{(1-m)s} \; |f(y)|}{|y|^s}\right)dy\\
        = 2^{p+1+N}\I{A_{j-m}}\frac{A\left(\frac{2^{(1-m)s} \; |f(y)|}{|y|^s}\right)dy}{2^{(1-m)N}A\left(\frac{|f(y)|}{|y|^s}\right)dy}A\left(\frac{|f(y)|}{|y|^s}\right)dy\\
        =2^{p+1+N}2^{(1-m)(1-N)}\alpha_{s,A}\left(2^{(1-m)s}\right)\I{A_{j-m}}A\left(\frac{|f(y)|}{|y|^s}\right)dy
        <\frac{1}{2}\I{A_{j-m}}A\left(\frac{|f(y)|}{|y|^s}\right)dy.
    \end{multline*}
    Finally we apply \cref{iterative process} to complete the proof.
\end{proof}\smallskip  

 \begin{proof}[\textbf{Proof of \Cref{Main Theorem Point Hardy}}]
     
     Set $D=\RR^N\setminus \lbrace 0 \rbrace$, take $f\in\test(D)$. Then $f\equiv0$ near $0$. First let us assume  $\liminf\limits_{\lambda\to\infty}\lambda^{\frac{1-N}{s}}\alpha_{s,A}(\lambda)=0$. Choose $r>0$ small enough, such that, $\mbox{supp}(f)\subseteq B(0,r)^c$. Then \cref{hardy lambda infty} implies $H_{N,s,A}(D)>0$.
     
     If, however, $\liminf\limits_{\lambda\to0+}\lambda^{\frac{1-N}{s}}\alpha_{s,A}(\lambda)=0$, we choose $r>0$ so small that $\mbox{supp}(f)\subseteq B(0,2^mr)^c$. Proceed exactly same as in the previous case and apply \cref{hardy lambda zero}. The restriction on $\mbox{supp}(f)$ ensures that the last term in the inequality is zero. Hence $H_{N,s,A}(\RR^N\setminus \lbrace 0 \rbrace)>0$.
 \end{proof}
We shall the notations as in the geometric decomposition of $D$, done after \cref{property3}, for \cref{Hardy in cube unbounded}. 
\begin{lemma}\label{Hardy in cube unbounded}
    For a fixed $z\in\partial D$, define $Q_k:=\cup_{\ell=k}^{\infty}A_\ell$, where $A_\ell$ are as in \cref{Ak set}. Let $\liminf\limits_{\lambda\to \infty}\alpha_{s,A}(\lambda)=0$. Then there exists $m=m\left(\mbox{Lip}(D),s,A\right)$ and $c=c\left(\mbox{Lip}(D),s,A\right)$ such that for any $f\in\test(D)$,
    \begin{equation*}
    \I{Q_m}A\left(\frac{|f(x)|}{\delta_x^s}\right)dx
    \leq
    c\left[\ \II{Q_m\times Q_m}
    A\left(\frac{|f(x)-f(y)|}{|x-y|^s}\right)\frac{dydx}{|x-y|^N}\ +\I{Q_z(\rho)\setminus Q_m}A\left(\frac{|f(x)|}{\delta_x^s}\right)dx\right].
    \end{equation*}
\end{lemma}
\begin{proof}
   From the hypothesis on $s$ and $A$, there exists $m\in\NN$ such that
   $$
   2^{2+p}(1+\mbox{Lip}(D))\alpha_{s,A}(2^{s+ms}(1+\mbox{Lip}(D))^{s})<\frac{1}{2}.  
   $$
   We fix this $m$. Let $l_1=2(1+\mbox{Lip}(D))\sqrt{N-1+(\mbox{Lip}(D)\sqrt{N-1}+2)^2}$, $l_2=2^{m-1}$ and  $G=G(f,Q_z(\rho);l_1,l_2)$. For $j\geq m$, set $E_1=A_j\cap Q^i_j$ and $E_2=A_{j-m}\cap Q^i_j$. \Cref{dist est2,measure} imply that we can apply \cref{property3} on $E_1,E_2$. Thus we get, utilizing the choice of $m$,
   \begin{multline*}
   \I{(A_j\cap Q^i_j)\cap G} A\left(\frac{|f(x)|}{\delta_x^s}\right)
   dx
   \leq 2^{1+p-m}\I{A_{j-m}\cap
       Q^i_j}A\left(2^{s+ms}(1+\mbox{Lip}(D))^{s}\frac{|f(y)|}{\delta_y^s}\right)dy\\
   =2^{2+p}(1+\mbox{Lip}(D))\I{A_{j-m}\cap
       Q^i_j}\frac{A\left(2^{s+ms}(1+\mbox{Lip}(D))^{s}\frac{|f(y)|}{\delta_y^s}\right)}{2^{1+m}(1+\mbox{Lip}(D))A\left(\frac{|f(y)|}{\delta_y^s}\right)}A\left(\frac{|f(y)|}{\delta_y^s}\right)dy\\
   \leq 2^{2+p}(1+\mbox{Lip}(D))\alpha_{s,A}(2^{s+ms}(1+\mbox{Lip}(D))^{s})\I{A_{j-m}\cap
       Q^i_j}A\left(\frac{|f(y)|}{\delta_y^s}\right)dy
   <\frac{1}{2}\I{A_{j-m}\cap Q^i_j}A\left(\frac{|f(y)|}{\delta_y^s}\right)dy.
   \end{multline*}
   Summing over $1\leq i \leq 2^{j(N-1)}$, we get
   $$
  \I{A_j\cap G} A\left(\frac{|f(x)|}{\delta_x^s}\right)
  dx
   \leq \frac{1}{2}\I{A_{j-m}}A\left(\frac{|f(y)|}{\delta_y^s}\right)dy.
   $$
   Applying \cref{iterative process} the lemma follows.
\end{proof}
\smallskip

\begin{proof}[\textbf{Proof of \Cref{Main Theorem Other Hardy}}]
    (1) Without loss of generality, we can assume $D=\RR^{N-1}\times\lbrace 0 \rbrace$. Let $z\in D$. First assume the case $\liminf\limits_{\lambda\to 0}\alpha_{s,A}(\lambda)=0$. Take $f\in\test(D)$. Note that for the choice of $D$ in this case, we are free to choose $\rho$ as large as we want and we still have $Q_z(\rho)\subset D$. Choose $\rho>0$ so that $\mbox{supp}(f)\subseteq Q_z(\rho)$. Then applying \cref{Hardy in cube} we get 
    \begin{multline*}
    \I{D} A\left(\frac{|f(x)|}{\delta_x^s}\right)dx
    =\I{Q_z(\rho)} A\left(\frac{|f(x)|}{\delta_x^s}\right)dx
    \leq c\II{Q_z(\rho)\times Q_z(\rho)}
    A\left(\frac{|f(x)-f(y)|}{|x-y|^s}\right)\frac{dxdy}{|x-y|^N}\\
    \leq c\II{D\times D}
    A\left(\frac{|f(x)-f(y)|}{|x-y|^s}\right)\frac{dxdy}{|x-y|^N}.
    \end{multline*}
    Now assume $\liminf\limits_{\lambda\to \infty}\alpha_{s,A}(\lambda)=0$. \Cref{Hardy in cube unbounded} implies 
    $$
    \I{0<x_N<\frac{\rho}{2^m}} A\left(\frac{|f(x)|}{\delta_x^s}\right)dx
    \leq c\II{0<x_N,y_N<\frac{\rho}{2^m}}
    A\left(\frac{|f(x)-f(y)|}{|x-y|^s}\right)\frac{dxdy}{|x-y|^N}
    + \I{x_N>\frac{\rho}{2^m}} A\left(\frac{|f(x)|}{\delta_x^s}\right)dx,
    $$ where $m,\rho$ is the integer considered in \cref{Hardy in cube unbounded}. Note that $m$ does not depend on $\rho$. So we choose $\rho$ big enough, so that $\mbox{supp}(f)\subseteq \lbrace x\in \RR^N \ | \ x_N<\frac{\rho}{2^m}\rbrace$ forcing the remainder term in the above equation to vanish.
    \smallskip
    
    (2) Let $f\in\test(D)$. We choose $R>0$ large enough so that $D^c\subseteq B(0,R)$. Now $D\cap B(0,R)$ is a bounded Lipschitz domain. A careful observation of the proof of \cref{Main theorem} reveals that, under the assumption $\liminf\limits_{\lambda\to0+}\alpha_{s,A}(\lambda)=0$, we can use the same technique to get a constant $c=c(A,s,N,D)$ such that 
    \begin{equation}\label{th4-2}
    \I{D\cap B(0,R)}A\left(\frac{|f(x)|}{\delta_x^s}\right)dx\leq c \II{D\cap B(0,R) \times D\cap B(0,R)}A\left(\frac{|f(x)-f(y)|}{|x-y|^s}\right)\frac{dxdy}{|x-y|^N},
    \end{equation}
    even if $f\notin\test(D\cap B(0,R))$.\\
    
    First we assume $\liminf\limits_{\lambda\to\infty}\lambda^{\frac{1-N}{s}}\alpha_{s,A}(\lambda)=\liminf\limits_{\lambda\to0+}\alpha_{s,A}(\lambda)=0$.\\
        We apply \cref{hardy lambda infty} to get, for any $r=R$,
        \begin{equation}\label{th4-3}
        \I{B(0,R)^c}A\left(\frac{|f(x)|}{\delta_x^s}\right)dx\leq
        c\II{B(0,R)^c\times B(0,R)^c}
        A\left(\frac{|f(x)-f(y)|}{|x-y|^s}\right)\frac{dydx}{|x-y|^N}.
        \end{equation}
        
        Adding \cref{th4-2,th4-3}, the claim follows.\\

         Now we assume $\liminf\limits_{\lambda\to0+}\lambda^{\frac{1-N}{s}}\alpha_{s,A}(\lambda)=0$. Note that in this case $\liminf\limits_{\lambda\to0+}\alpha_{s,A}(\lambda)=0$ is implied. Consider the extension of $f$ by $0$ to whole of $\RR^N$.
        Applying \cref{hardy lambda zero} with $r=\frac{R}{2^m}$ and then proceeding similarly as above, the claim follows.\\ 
        
    Finally, assume $\liminf\limits_{\lambda\to \infty}\alpha_{s,A}(\lambda)=0$. Denote $\forall\ \eta>0$, $D_\eta:=\lbrace x\in D \ | \ \mbox{dist}(x,\partial D)<\eta \rbrace$. An application of \cref{Hardy in cube unbounded}, as \cref{Hardy in cube} was used in the proof of \cref{Main theorem}, gives
    \begin{multline*}\label{th4-1}
    \I{D_\varepsilon}A\left(\frac{|f(x)|}{\delta_x^s}\right)dx
    \leq
    c\left[\ \II{D_\varepsilon\times D_\varepsilon}
    A\left(\frac{|f(x)-f(y)|}{|x-y|^s}\right)\frac{dydx}{|x-y|^N}\ +\I{D_{2\varepsilon}\setminus D_\varepsilon}A\left(\frac{|f(x)|}{\delta_x^s}\right)dx\right]\\
    \leq
    c\left[\ \II{D_\varepsilon\times D_\varepsilon}
    A\left(\frac{|f(x)-f(y)|}{|x-y|^s}\right)\frac{dydx}{|x-y|^N}\ +\I{D\setminus D_\varepsilon}A\left(\frac{|f(x)|}{\delta_x^s}\right)dx\right],
    \end{multline*}
   for some fixed $\varepsilon>0$ and a constant $c=c(s,N,A,D)>0$ and or any $f\in\test(D)$.
   Now set $\Om:=D\setminus D_\varepsilon$ and let $R>0$ be such that $\Om^c\subseteq B(0,R)$. Before proceeding further note that if, for some constant $c=c(s,N,A,D)>0$ and for all $f\in\test(D)$, we can show
   $$
   \I{\Om}A\left(\frac{|f(x)|}{\delta_x^s}\right)dx
   \leq
   c\ \II{\Om\times \Om}
   A\left(\frac{|f(x)-f(y)|}{|x-y|^s}\right)\frac{dydx}{|x-y|^N},
   $$
   then the proof will be complete. Set $A_j:=B(0,2^{j}R)\setminus B(0,2^{j-1}R)$ and $A_0:=B(0,R)\cap \Om$. Then for $k\in\NN$, $\mathcal{L}^N(A_k)=\mathcal{L}^N(B(0,1))R^N2^{(k-1)N}(2^N-1)$ and for any $x\in A_k$, $k>1$, $(2^{k-1}-1)R\leq\delta_x\leq 2^{k}R$. Whereas $\mathcal{L}^N(A_0)=\mathcal{L}^N(B(0,R))-\mathcal{L}^N(\Om^c)$, for $x\in A_0\cup A_1$, $\varepsilon\leq\delta_x\leq R$ and for $x\in A_1$, $\varepsilon\leq \delta_x\leq 2R$. Now let $x\in A_j,\ y\in A_{j+m}$. Then for any $j\geq0$, $\mbox{dist}(x,y)\leq2^{j+m}R=2^m\frac{2^{j}}{2^{j-1}-1}$. So it is clear that there exists $l_1,l_2>0$ such that the hypotheses of \cref{property3} are satisfied. So, we get
\begin{multline*}
   \I{A_j\cap G}A\left( \frac{|f(x)|}{\delta_x^s}\right)
   dx
   \leq\frac{2^{p+1}\mathcal{L}^{N}(A_j)}{\mathcal{L}^{N}(A_{j+m})}\I{A_{j+m}}A\left(\frac{\sup\lbrace \delta_x^s \ | \ x\in A_{j+m}\rbrace \; |f(y)|}{\inf\lbrace \delta_x^s \ | \ x\in A_{j}\rbrace \;\delta_y^s}\right)dy\\
   \leq c(R,\mathcal{L}^N(\Om^c),N)\ 2^{-mN}\I{A_{j+m}}A\left(\frac{\sup\lbrace \delta_x^s \ | \ x\in A_{j+m}\rbrace \; |f(y)|}{\inf\lbrace \delta_x^s \ | \ x\in A_{j}\rbrace \;\delta_y^s}\right)dy\\
   \leq c\ 2^{-mN}\I{A_{j+m}}A\left(\frac{c(\varepsilon,R) 2^{ms}\; |f(y)|}{\delta_y^s}\right)dy.
   \end{multline*}
   Note that $\liminf\limits_{\lambda\to \infty}\alpha_{s,A}(\lambda)=0$ implies $\liminf\limits_{\lambda\to \infty}\lambda^\frac{1-N}{s}\alpha_{s,A}(\lambda)=0$. Choosing $m$ sufficiently large and then applying this, we get
   
   $$
   \I{A_j\cap G}A\left( \frac{|f(x)|}{\delta_x^s}\right)
   dx
   \leq \I{A_{j+m}}A\left( \frac{|f(y)|}{\delta_y^s}\right)
   dy.
   $$
   The proof follows as an application of \cref{iterative process}.
\end{proof}
\bigskip

        \section{\texorpdfstring{$RFOPI(s,A)$}{} and \texorpdfstring{$FOPI(s,A)$}{} on unbounded domains}\label{section poincare}
        In this section we study $RFOPI(s,A)$ and $FOPI(s,A)$ for general
        domains in $\RR^N$. 
        \begin{proposition}\label{1-dim poincare}
            Let $D\subseteq\RR$ be an open set and $\lim\limits_{\lambda\to 0}\alpha_{s,A}(\lambda)=0$. Then $P^1_{N,s,A}(D)>0$
            if and only if $BC(D)<\infty$.
        \end{proposition}
        \begin{proof}
            Suppose $BC(D)<\infty$. There exists a countable number of disjoint open intervals, say $I_k$, such	that $D=\cup_{k=1}^{\infty}I_k$. Let $f\in\test(D)\setminus\lbrace 0\rbrace$.
            Using \cref{Main theorem} followed by \cref{BC lemma}, we get
            \begin{multline*}
            \II{D\times D} A\left(\frac{|f(x)-f(y)|}{|x-y|^s}\right)\frac{dxdy}{|x-y|}
            \geq\sum_{k=1}^\infty \ \II{I_k\times I_k}
            A\left(\frac{|f(x)-f(y)|}{|x-y|^s}\right)\frac{dxdy}{|x-y|}\\
            \geq\sum_{k=1}^\infty P^1_{1,s,A}(I_k)\I{I_k}A(|f(x)|)dx
            \geq \min \lbrace BC(D)^{-sp},1 \rbrace P^1_{1,s,A}((0,1))\sum_{k=1}^\infty \I{I_k}A(|f(x)|)dx\\
            = \min \lbrace BC(D)^{-sp},1 \rbrace P^1_{1,s,A}((0,1))\I{D}A(|f(x)|)dx,
            \end{multline*}
            which implies $P^1_{1,s,A}(D)\geq c P^1_{1,s,A}((0,1))>0$.
            The other part follows from \cref{monotone-dialtion-fbc}.
        \end{proof}
        The next lemma is an important change of variable formula for the  fractional Orlicz seminorm. 
        \begin{lemma}\label{LS lemma}
            Let $0<s<1$ and $D\subseteq\RR^N$ be a measurable set.
            Then for any $f\in \test(D)$,
            \begin{multline*}
            2\II{ D \times D }
            A\left( \frac{|f(x)-f(y)|}{|x-y|^s}\right)  \frac{dxdy}{|x-y|^N}
            \\
            =\I{\omega\in\mathbb{S}^{N-1}}
            \I{x\in\omega^\perp}
            \I{\ell\in L_ D (x,\omega)}
            \I{t\in L_ D (x,\omega)}A\left(
            \frac{|f(x+\ell\omega)-f(x+t\omega)|}{|\ell-t|^s}\right)  \frac{dtd\ell
                d\sigma(x)d\sigma(\omega)}{|\ell-t|},
            \end{multline*}
            where $L_{D}(x,
            \omega) := \lbrace t
            \ |  \ x+t\omega \in D \rbrace$.
        \end{lemma}
        \begin{proof}
            For a fixed $f\in\test( D )$ and $x\in D $, consider the change of variable $y=x+z$ followed by another change of variable to polar coordinates. We then get
            \begin{multline*}
            2\II{ D \times D }
            A\left( \frac{|f(x)-f(y)|}{|x-y|^s}\right)  \frac{dxdy}{|x-y|^N}
            =2\I{x\in D } \I{z\in D -x}A\left( \frac{|f(x)-f(x+z)|}{|z|^s}\right)
            \frac{dzdx}{|z|^N}\\
            =2\I{x\in D } \I{\omega\in\mathbb{S}^{N-1}}\I{r\in\lbrace r>0 \ | \ x+r\omega\in
                D \rbrace }A\left( \frac{|f(x)-f(x+r\omega)|}{r^s}\right)
            \frac{drd\sigma(\omega)dx}{r}\\
            =\I{\omega\in\mathbb{S}^{N-1}}\I{x\in D } \I{h\in L_ D (x,\omega)}A\left(
            \frac{|f(x)-f(x+h\omega)|}{|h|^s}\right)  \frac{dhdxd\sigma(\omega)}{|h|}.
            \end{multline*}
            Now, for a fixed $\omega\in\mathbb{S}^{N-1}$ we can write $x=(x-(x \cdot
            \omega)\omega)+(x \cdot \omega)\omega$. Using this we can split the
            integral over $x\in D $ by considering the change of variable
            $\ell=x\cdot\omega$ and $z=x-(x \cdot \omega)\omega$, where $\ell\in L_ D
            (z,\omega)$ and $z\in\omega^\perp$. So we get
            \begin{multline*}
            2\II{ D \times D }
            A\left( \frac{|f(x)-f(y)|}{|x-y|^s}\right)  \frac{dxdy}{|x-y|^N}
            \\
            =\I{\omega\in\mathbb{S}^{N-1}}
            \I{z\in\omega^\perp}
            \I{\ell\in L_ D (z,\omega)}
            \I{h\in L_ D (x,\omega)}A\left(
            \frac{|f(z+\ell\omega)-f(z+(\ell+h)\omega)|}{|h|^s}\right)  \frac{dhd\ell
                d\sigma(z)d\sigma(\omega)}{|h|}\\
            =\I{\omega\in\mathbb{S}^{N-1}}
            \I{z\in\omega^\perp}
            \I{\ell\in L_ D (z,\omega)}
            \I{t\in L_ D (z+\ell\omega,\omega)+\ell}A\left(
            \frac{|f(z+\ell\omega)-f(z+t\omega)|}{|\ell-t|^s}\right)  \frac{dhd\ell
                d\sigma(z)d\sigma(\omega)}{|\ell-t|}.
            \end{multline*}
            We conclude the proof of this lemma with the observation that $L_ D
            (z,\omega)=L_ D (z+\ell\omega,\omega)+\ell.$
            This finishes the proof of the theorem. \end{proof}\smallskip

        \begin{proof}[\textbf{Proof of \Cref{Sufficient condition}}]
            Let $f\in\test( D )$.  (1) Applying \cref{LS lemma} we obtain
            \begin{multline*}
            2\II{ D \times D }
            A\left( \frac{|f(x)-f(y)|}{|x-y|^s}\right)  \frac{dxdy}{|x-y|^N}
            \\
            =\I{\omega\in\mathbb{S}^{N-1}}
            \I{x\in\omega^\perp}
            \I{\ell\in L_ D (x,\omega)}
            \I{t\in L_ D (x,\omega)}A\left(
            \frac{|f(x+\ell\omega)-f(x+t\omega)|}{|\ell-t|^s}\right)  \frac{dtd\ell
                d\sigma(x)d\sigma(\omega)}{|\ell-t|}\\
            \geq \I{\omega\in\Sigma}
            \I{x\in\omega^\perp}
            \I{\ell\in L_ D (x,\omega)}
            \I{t\in L_ D (x,\omega)}A\left(
            \frac{|f(x+\ell\omega)-f(x+t\omega)|}{|\ell-t|^s}\right)  \frac{dtd\ell
                d\sigma(x)d\sigma(\omega)}{|\ell-t|}\\
            =\sum_{k=1}^{\infty} \I{\omega\in\Sigma}
            \I{x\in\omega^\perp}
            \I{\ell\in L_ D (x,\omega)}
            \I{t\in I_k}A\left(
            \frac{|f(x+\ell\omega)-f(x+t\omega)|}{|\ell-t|^s}\right)  \frac{dtd\ell
                d\sigma(x)d\sigma(\omega)}{|\ell-t|}\\
            \geq \sum_{k=1}^{\infty}\I{\omega\in\Sigma}
            \I{x\in\omega^\perp}
            P^1_{1,s,A}(I_k)\I{t\in I_k}A(|f(x+t\omega)|)dt
            d\sigma(x)d\sigma(\omega),
            \end{multline*}
            where $L_ D (x,\omega)=\cup_{k=1}^\infty I_k$, for mutually disjoint family of intervals $\lbrace I_k \rbrace_k$.
            From \cref{BC lemma} we have, for any $k$,
            $$
            P^1_{1,s,A}(I_k)\geq \min\lbrace BC(L_D(x,\omega))^{-sp},1\rbrace P^1_{1,s,A}((0,1)).
            $$
            We can now use the hypothesis of the lemma to get a constant $c=c(D)>0$ such that $c<\min\lbrace BC(L_D(x,\omega))^{-sp},1\rbrace$. Again as an application of \cref{Main theorem} $P^1_{1,s,A}((0,1))>0$. Using these we get
            \begin{multline*}
            \II{D\times D}
            A\left( \frac{|f(x)-f(y)|}{|x-y|^s}\right)  \frac{dxdy}{|x-y|^N}
            \geq c \I{\omega\in\Sigma}
            \I{x\in\omega^\perp}
            \I{t\in L_ D (x,\omega)}A(|f(x+t\omega)|)dt
            d\sigma(x)d\sigma(\omega)\\
            \geq c \I{\omega\in\Sigma}
            \I{x\in D }A(|f(x)|)dxd\sigma(\omega)
            =c\I{D}A(|f(x)|)dx.
            \end{multline*}
            which implies that $P^1_{N,s,A}( D )>0.$ This completes the first part of the proof.\smallskip
            
            (2) We use the positivity, increasing property of $A$,
            the
            fact that $f\equiv0$ on $ D ^c$ and suitable change of variable in the
            following calculation.
            \begin{multline*}
            \II{\RR^N\times\RR^N}A\left( \frac{|f(x)-f(y)|}{|x-y|^s}\right)
            \frac{dydx}{|x-y|^N}=\frac{1}{R^N}\II{\RR^N\times\RR^N}A\left(
            \frac{R^s|f(\frac{x}{R})-f(\frac{y}{R})|}{|x-y|^s}\right)
            \frac{dydx}{|x-y|^N}\\
            \geq \frac{1}{R^N}\I{x\in\RR^N}\I{y\in B(x,R)\cap  D ^c}A\left(
            \frac{R^s|f(\frac{x}{R})|}{|x-y|^s}\right)  \frac{dydx}{|x-y|^N}\geq
            \frac{\mathcal{L}^{N}(B(x,R)\cap  D ^c)}{R^{N}}\I{\RR^N}A\left(|f\left(\frac{x}{R}\right)|\right)dx\\
            =\frac{c_1}{R^{N}}\I{\RR^N}A\left(|f(x)|\right)dx=\frac{c_1}{R^{N}}\I{D}A\left(|f(x)|\right)dx.
            \end{multline*}
            This shows $P^2_{N,s,A}( D )>0$ and the proof of the second part follows.\smallskip
            
            (3) This is a straightforward use of \cref{LS lemma} as in the first part of the proof.
        \end{proof}
        
        As an application of \cref{Sufficient condition} we discuss here
        some examples of some unbounded domains for which Poincar\'e
        inequality holds true, provided $\lim\limits_{\lambda\to 0^+}\alpha_{s,A}(\lambda)=0$ for the first condition (see also \cite{Ind,Ind-Roy}). Although verification of the claims in these examples is straight forward, we work it out in the second example.
        \smallskip
        
        \begin{example}[Domain between graphs of two functions] Let
            $f_i:\RR^{N-1}\to[m,M]$ ($i=1,2$) be two bounded continuous function such
            that $f_1<f_2$. We then define $D$ as
            $$
            D=\lbrace (\tilde{x},x_N)\in\RR^N\ | \ f_1(\tilde{x})<x_N<f_2(\tilde{x})\rbrace .
            $$
            It is easy to verify that all the hypotheses of \cref{Sufficient condition} are true.
        \end{example}
        
        \begin{example}[Countable union of parallel strips]
            Let $D:=I\times\RR\subseteq\RR^2$, where $I:=\cup_{k=1}^\infty I_k$ and $I_k$'s are disjoint intervals with $\mbox{dist}(I_i,I_j)\geq\gamma$ for a constant $\gamma>0$, where $I_i,\;I_j$ are any two distinct intervals. Also assume that $\sup\limits_{k}(\mathcal{L}^1(I_k))=M<\infty.$ Then it is easy to check that the second condition of \cref{Sufficient condition} holds. This gives $P^2_{1,s,A}(I)>0$, for any $s\in(0,1)$. Now choose  $\Sigma=\{\omega_\theta:=(cos\theta,sin\theta)\ | \ \theta\in(0,\pi/4)\}$. Then for $x\in\omega_\theta^\perp$, the set $L_D(x,\omega_\theta)$ can be expressed as the countable union of disjoint intervals, i.e. $L_D(x,\omega_\theta)=\cup_{k=1}^\infty J_k(x,\omega_\theta)$. Observe that $\mathcal{L}^1(J_k(x,\omega_\theta))=\frac{ \mathcal{L}^1(I_k)}{cos\theta}$ for any $k$, which is independent of $x$. Thus, we have for some constant $c=c(x,\omega_\theta)$,
            \begin{equation}\label{LD}
            L_D(x,\omega_\theta)=c+(\sec\theta) I.
            \end{equation}
            This gives $\sup\limits_{\omega_\theta
                \in\Sigma, x\in
                \omega_\theta^\perp}BC(L_D(x,\omega_\theta))\leq M\sqrt{2}$ and which shows that the first condition holds. Consequently $RFOPI(s,A)$ holds in $D$ when $\lim\limits_{\lambda\to0+}\alpha_{s,A}(\lambda)=0$. For third condition, Using (3) of \cref{monotone-dialtion-fbc,LD} we obtain, since $\theta\in(0,\pi/4)$,
            $$P^2_{1,s,A}(L_D(x,\omega_\theta))\geq\frac{P^2_{1,s,A}(I)}{(cos\theta)^{sp}}\geq 2^{\frac{sp}{2}}P^2_{1,s,A}(I).$$
            So, uniform $FOPI(s,A)$ holds for the collection $\lbrace L_D(x,\omega_\theta)\rbrace_{\omega_\theta\in\Sigma,x\in\omega_\theta^\perp}$. Hence the domain $D$ is of type $LS(s,A)$ for any $s\in(0,1)$. Consequently $FOPI(s,A)$ holds in $D$ for any $s\in(0,1)$.
        \end{example}
        
        \begin{example}[Concentric annulus] The following
            domain satisfies condition 2 for all $s\in(0,1)$ of \cref{Sufficient condition}:
            $$
            D=\bigcup_{k=1}^\infty B(0,2k)\setminus B(0,2k-1).
            $$
        \end{example}
        
        \begin{example}[Domain with holes around points of $\ZZ\times\ZZ$] The following domain satisfies the condition 2 of \cref{Sufficient condition}:
            $$
            D=\RR^2\setminus\left(\bigcup\limits_{x\in\mathbb{Z}^2}B\left(x,1/10\right)\right).
            $$
        \end{example}
        
        \bigskip
        
        \section*{Acknowledgement} Research work of first author is funded by Matrics grant  (MTR/2020/000594). The work of the second author is supported by CSIR(09/092(0940)/2015-EMR-I). Research work of third author is funded by Matrics grant (MTR/2019/000585) and Inspire grant  (IFA14-MA43) of Department of Science and Technology (DST).

    \end{document}